\documentclass[12pt]{amsart}

\usepackage{fullpage}
\usepackage[colorlinks=true,citecolor=black,linkcolor=black,urlcolor=blue]{hyperref}

\numberwithin{equation}{section}

\usepackage{lipsum}
\usepackage{subcaption}

\usepackage{verbatim}
\usepackage{amsmath}
\usepackage{amssymb}
\usepackage{amsthm}
\usepackage{url}
\usepackage{subcaption}
\usepackage{mleftright}
\usepackage{color,hyperref}
\usepackage{booktabs}
\usepackage{xcolor}

\usepackage{tikz,graphicx}
\tikzstyle{treenode}=[circle, draw, inner sep=0pt, minimum size=12pt]
\newcommand{\treenode}{\node[treenode]}
\tikzstyle{leafnode}=[fill, circle, draw, inner sep=0pt, minimum size=5pt]
\newcommand{\leafnode}{\node[leafnode]}
\tikzstyle{vertex}=[circle, draw, inner sep=0pt, minimum size=4pt]

\newcommand{\vol}{\operatorname{vol}}

\definecolor{sienna}{RGB}{136,45,23}

\newcommand\bbinom[2]
{\mathchoice{\left(\kern-0.48em{\binom{#1}{#2}}\kern-0.48em\right)}
            {\bigl(\kern-0.3em{\binom{#1}{#2}}\kern-0.3em\bigr)}
            {\bigl(\kern-0.3em{\binom{#1}{#2}}\kern-0.3em\bigr)}
            {\bigl(\kern-0.3em{\binom{#1}{#2}}\kern-0.3em\bigr)}}

\newtheorem{theorem}{Theorem}[section] 
\newtheorem{proposition}[theorem]{Proposition} 
\newtheorem{lemma}[theorem]{Lemma} 
\newtheorem{corollary}[theorem]{Corollary}
\newtheorem{conjecture}[theorem]{Conjecture}

\newtheorem*{theorem*}{Theorem}

\theoremstyle{remark}
\newtheorem{remark}[theorem]{Remark}

\newtheorem{example}[theorem]{Example}

\theoremstyle{definition}
\newtheorem{definition}[theorem]{Definition}

\newcommand{\defn}[1]{\emph{\color{blue} #1}}

\def\RR{\mathbb{R}}

\def\cA{\mathcal{A}}
\def\cC{\mathcal{C}}
\def\cD{\mathcal{D}}
\def\cF{\mathcal{F}}
\def\cH{\mathcal{H}}
\def\cI{\mathcal{I}}
\def\cN{\mathcal{N}}
\def\cO{\mathcal{O}}

\def\cS{\mathcal{S}}
\def\cT{\mathcal{T}}

\def\cW{\mathcal{W}}

\DeclareMathOperator{\ASC}{ASC}

\DeclareMathOperator{\s}{s}
\DeclareMathOperator{\car}{car}
\DeclareMathOperator{\oru}{oru}
\newcommand{\sizes}[1][\s]{|#1|}
\newcommand{\settrees}[1][\s]{\cT_{#1}}
\newcommand{\spermcombi}[1][\s]{\text{Perm}_{#1}} 
\newcommand{\setperms}[1][\s]{\cW_{#1}} 

\newcommand{\fpol}[1][G]{\cF_{#1}} \newcommand{\oruga}[1][n]{\oru_{#1}} 
\newcommand{\Gs}[1][\s]{\oru(#1)} 

\DeclareMathOperator{\route}{R}
\newcommand{\routep}[1]{\route[#1]} 
\newcommand{\prefix}[2]{{#1}_{[{#2}]}} 
\newcommand{\conv}{\text{conv}}
\DeclareMathOperator{\supp}{\text{supp}} 
\DeclareMathOperator {\inv}{inv}
\DeclareMathOperator {\indeg}{indeg}
\DeclareMathOperator {\outdeg}{outdeg}
\DeclareMathOperator{\height}{h}

\newcommand{\triangDKK}[1][G]{\text{Triang}_{DKK}({#1} , \preceq\nolinebreak)} 
\newcommand{\cliques}[1][G]{\text{Cliques}(#1, \preceq)}
\newcommand{\maxcliques}[1][G]{\text{MaxCliques}(#1, \preceq)} 
\newcommand{\subdivCay}[1][\s] {\text{Subdiv}_{\square}(#1)} 
\newcommand{\spermgeom}[1][\height]{\text{Perm}_{\s}(#1)} 
\newcommand{\troparr}[1][\height]{\cH_{\s}(#1)} 
\newcommand{\Gale}[1][k,m]{\textit{Gale}(#1)} 

\title[Realizing the $\s$-Permutahedron via Flow Polytopes]{Realizing the $\s$-Permutahedron via Flow Polytopes}

\author[Gonz\'{a}lez D'Le\'{o}n]{Rafael S. Gonz\'alez D'Le\'on}
\address[R. S. Gonz\'{a}lez D'Le\'{o}n]{Department of Mathematics and Statistics, Loyola University of Chicago, Chicago, USA}
\email{rgonzalezdleon@luc.edu}
\urladdr{https://dleon.combinatoria.co/}

\author[Morales]{Alejandro H. Morales}
\address[A. H. Morales]{Department of Mathematics and Statistics, UMass Amherst, U.S.A}
\email{ahmorales@math.umass.edu}
\urladdr{https://people.math.umass.edu/~ahmorales/}

\author[Philippe]{Eva Philippe}
\address[E. Philippe]{Sorbonne Université and Université de Paris, CNRS, IMJ-PRG, F-75005 Paris, France }
\email{eva.philippe@imj-prg.fr}
\urladdr{https://perso.imj-prg.fr/eva-philippe/}

\author[Tamayo Jim\'enez]{Daniel Tamayo Jim\'enez}
\address[D. Tamayo Jim\'enez]{Universit\'e Paris-Saclay, GALaC, Gif-sur-Yvette, France.}
\urladdr{https://sites.google.com/view/danieltamayo22/}

\author[Yip]{Martha Yip}
\address[M. Yip]{Department of Mathematics, University of Kentucky, Lexington KY, USA}
\email{martha.yip@uky.edu}
\urladdr{https://www.ms.uky.edu/~myip/}

\thanks{Partially supported by the MATH-AMSUD project 22-MATH-01 ALGonCOMB. AHM is partially supported by NSF grants DMS-1855536 and DMS-22030407, EP is supported by grants ANR-17-CE40-0018 and ANR-21-CE48-0020 of the French National Research Agency ANR (projects CAPPS and PAGCAP), DTJ is supported by grant ANR-21-CE48-0020 of the French National Research Agency ANR (project PAGCAP), and MY is partially supported by Simons collaboration grant 964456.}

\keywords{$\s$-weak order, $\s$-decreasing trees, Stirling $\s$-permutations, flow polytopes, geometric realization, polyhedral subdivision, Cayley trick.}

\usepackage[backend=bibtex]{biblatex}
\begin{document}

\maketitle

\begin{abstract}
Ceballos and Pons introduced the $\s$-weak order on $\s$-decreasing trees, for any weak composition $\s$. They proved that it has a lattice structure and further conjectured that it can be realized as the $1$-skeleton of a polyhedral subdivision of a polytope. 
We answer their conjecture in the case where $\s$ is a strict composition by providing three geometric realizations of the $\s$-permutahedron. 
The first one is the dual graph of a triangulation of a flow polytope of high dimension. The second one, obtained using the Cayley trick, is the dual graph of a fine mixed subdivision of a sum of hypercubes that has the conjectured dimension. The third one, obtained using tropical geometry, is the $1$-skeleton of a polyhedral complex for which we can provide explicit coordinates of the vertices and whose support is a permutahedron as conjectured.
\end{abstract}

\tableofcontents

\section{Introduction}
The starting point of this work is a conjecture of Ceballos and Pons (\cite[Conjecture 1]{CP20}, also Conjecture~\ref{conj:s-permutahedron} below) stating that a certain combinatorial complex on $\s$-decreasing trees can be geometrically realized as a polyhedral subdivision of a polytope. 
The family of $\s$-decreasing trees is parameterized by \defn{weak compositions} $\s=(s_1, \ldots, s_n)$ where $s_i$ are nonnegative integers for $i=1,2,\dots, n$.
Ceballos and Pons~\cite{CP20, CP22} showed that for every $\s$, the set of $\s$-decreasing trees admits a lattice structure called the $\s$-weak order. 
In the special case when $\s=(1, \ldots, 1)$, the set of $\s$-decreasing trees is in bijection with the set of permutations of $[n]:=\{1,\ldots,n\}$, and the $\s$-weak order is the classical (right) weak order on the permutations of $[n]$.  
In the same way that the weak order on permutations restricts to the lattice on Catalan objects introduced by Tamari in \cite{T62}  (as implied by a classical bijection of Stanley \cite[Section 1.5]{S12}), Ceballos and Pons show in \cite[Theorem 2.2]{CP22} that the $\s$-weak order restricts to the $\s$-Tamari lattice. 
The $\s$-Tamari lattice was first introduced by Pr\'eville--Ratelle and Viennot \cite{PRV17} as the $\nu$-Tamari lattice on grid paths weakly above the path $\nu=NE^{s_n}\ldots NE^{s_1}$  
(see \cite[Theorem 3.5]{CP20} for the isomorphism between the $\nu$-Tamari and the $\s$-Tamari lattices). It is a further generalization of the $m$-Tamari lattice, that is the case when $\s=(m, \dots, m)$, introduced by Bergeron and Pr\'eville--Ratelle in \cite{BPR12} to study the Frobenius characteristic of the space of higher diagonal coinvariant spaces.

\subsection{\texorpdfstring{$\s$}{s}-decreasing trees, \texorpdfstring{$\s$}{s}-weak order and the \texorpdfstring{$\s$}{s}-permutahedron}

Let  $\s=(s_1,\ldots,s_n)$ be a weak composition. An \defn{$\s$-decreasing tree} $T$ is a planar rooted tree on $n$ internal vertices (called nodes), labeled by $[n]$, such that the node labeled $i$ has $s_i+1$ children and any descendant $j$ of $i$ satisfies $j<i$. We denote by \defn{$T_0^i,\ldots,T_{s_i}^i$} the subtrees of node $i$ from left to right.

We denote by \defn{$\settrees$} the set of all $\s$-decreasing trees. 
Note that the value of $s_1$ is inconsequential for determining the combinatorial properties of $\settrees$, so without loss of generality we may assume that $s_1=1$ throughout this article.
It is known (see for example \cite[Section 5.1]{CGDL19} and Corollary~\ref{cor:volume is number of trees}) that the number of $\s$-decreasing trees is given by 
\begin{align}
    \# \settrees &= (1+s_n)(1+s_n+s_{n-1})\cdots  (1+s_n+s_{n-1}+\cdots + s_2), \label{eq: formula s decreasing trees} 
\end{align}
which can be viewed as a generalization of the factorial numbers since this formula reduces to $n!$ in the case $\s=(1, \ldots, 1)$. 

Let $T$ be an $\s$-decreasing tree. We denote by \defn{$\inv(T)$} the multiset of \defn{tree-inversions} of $T$ formed by pairs $(c,a)$ with multiplicity (also called cardinality)
\[\defn{$\#_T(c,a)$} =\left\{\begin{array}{ll}
    0, & \text{ if } a \text{ is left of } c, \\
    i, & \text{ if } a \in T^c_{i}, \\
    s_c, & \text{ if } a \text{ is right of } c,
    \end{array}\right.\]
for all $1\leq a < c \leq n$.

In \cite[Definition 2.5]{CP20} Ceballos and Pons introduced the \defn{$\s$-weak order $\trianglelefteq$} on $\settrees$  and showed in \cite[Theorem 1.21]{CP22} that it has the structure of a lattice. For $\s$-decreasing trees $R$ and $T$ we define $R \trianglelefteq T$ if $\inv(R)\subseteq \inv(T)$.

To understand the cover relations in the $\s$-weak order we define the notion of ascents. An \defn{ascent} on an $\s$-decreasing tree $T$ is a pair $(a,c)$ satisfying
\begin{enumerate}
    \item $a\in T^c_{i}$ for some $0\leq i < s_c$,
    \item if $a < b < c$ and $a\in T^b_{i}$, then $i=s_b$,
    \item if $s_a>0$, then $T^a_{s_a}$ consists of only one leaf.
\end{enumerate}

Similarly, a \defn{descent} of $T$ is a pair $(a,c)$ such that $a\in T_i^c$ for some $0 < i \leq s_c$, if $a<b<c$ and $a\in T_j^b$ then $j=0$, and if $s_a>0$ then $T_0^a$  consists of only one leaf. 
The notions of ascents and descents on $\s$-decreasing trees generalize the same concepts from classical permutations, see Lemma~\ref{lem:ascdesc}.

A cover relation of the $\s$-weak order can be seen either as an \defn{$\s$-tree rotation} along an ascent $(a,c)$ (see~\cite[Definition 1.30]{CP22}) or equivalently as taking the transitive closure of the multiset of inversions obtained from $\inv(T)$ after increasing $\#_T(c,a)$ by 1 (\cite[Theorem~1.32]{CP22}). If $A$ is a subset of ascents of $T$, we denote by \defn{$T+A$} the $\s$-decreasing tree whose inversion set is given by the transitive closure of $\inv(T)+A$.

\begin{definition}[Definition 4.1 \cite{CP20}]
\label{def:s-permutahedron}
The (combinatorial) \defn{$\s$-permutahedron}, denoted $\spermcombi$, is the combinatorial complex with faces $(T,A)$ where $T$ is an $\s$-decreasing tree and $A$ is a subset of ascents of $T$. 
\end{definition}
In the $\s$-permutahedron, the face $(T,A)$ is contained in $(T', A')$ if and only if $[T, T+A]\subseteq [T',T'+A']$ as intervals in the $\s$-weak order.
In particular, the vertices of $\spermcombi$ are the $\s$-decreasing trees and the edges correspond to $\s$-tree rotations. \\

It turns out that when $\s$ is a (strict) \defn{composition}, that is $s_i>0$ for all $i$, the properties of the $\s$-weak order and the $\s$-permutahedron can also be described in terms of \defn{Stirling $\s$-permutations}. 
These are multipermutations of $[n]$ avoiding the pattern $121$ (a number $j$ somewhere in between two occurrences $i$ with $i<j$) and with $s_i$ occurrences of $i$ for each $i\in [n]$. These multipermutations generalize the family of permutations (the case when $\s=(1,\dots,1)$) and the family of Stirling permutations (the case when $\s=(2,\dots, 2)$) initially introduced by Gessel and Stanley in \cite{SG78}. A further generalization of Stirling permutations (to the case when $\s=(m,\dots,m)$) was studied by Park in \cite{P94-1,P94-2,P94-3}. Further study of combinatorial formulas and statistics on $\s$-Stirling permutations such as descents, ascents and plateaux have been carried out by many other authors (see for example \cite{B09,JKP11,KP11}). We refer the reader to Gessel's note in \cite{G20} which includes a  list of articles on the family of $\s$-Stirling permutations.

Figure \ref{fig:adjacency_graph_121_classic} shows the Hasse diagram of the $\s$-weak order for the case $\s=(1,2,1)$.
The vertices are indexed by $\s$-decreasing trees and Stirling $\s$-permutations.

\begin{figure}[ht!]
    \centering
    \includegraphics[scale=0.7]{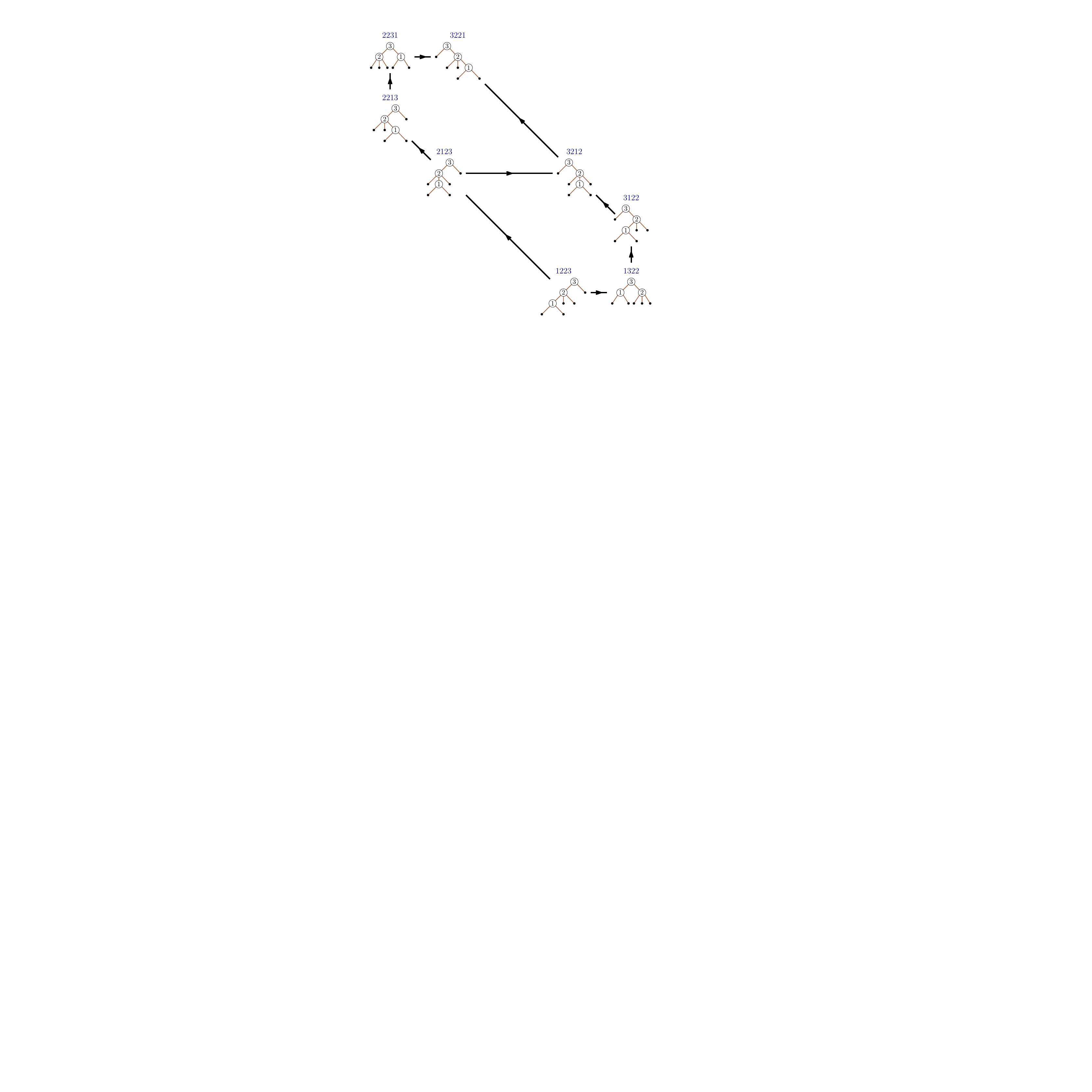}
    \caption{The $\s$-weak order ($1$-skeleton of the $\s$-permutahedron) for the case $\s=(1,2,1)$. The vertices are indexed by $\s$-decreasing trees and Stirling $\s$-permutations.}
 \label{fig:adjacency_graph_121_classic}
\end{figure}

\subsection{Geometric realizations of the \texorpdfstring{$\s$}{s}-permutahedron}

From Figure \ref{fig:adjacency_graph_121_classic} the reader can already appreciate how the $\s$-permutahedron may be geometrically realizable. 
Ceballos and Pons posed the following conjecture on realizations of $\spermcombi$.

\begin{conjecture}[{\cite[Conjecture 1]{CP20}}]\label{conj:s-permutahedron}
Let $\s$ be a weak composition. 
The $\s$-permutahedron $\spermcombi$ can be realized as a polyhedral subdivision of a polytope which is combinatorially isomorphic to the zonotope $\sum_{1\leq i < j \leq n} s_j [\mathbf{e_i},\mathbf{e_j}]$, where $(\mathbf{e_i})_{1\leq i\leq n}$ is the canonical basis of $\RR^n$ and $[\mathbf{e_i},\mathbf{e_j}]$ indicates the convex hull of $\mathbf{e_i}$ and $\mathbf{e_j}$.
\end{conjecture}

Our goal in the present article is to provide solutions to Conjecture \ref{conj:s-permutahedron} in the case when $\s$ is a strict composition (see Theorem~\ref{thm:all_realizations}). 
We will use techniques similar to those that were previously employed for realizing the $\s$-associahedron, which is a closely related combinatorial complex whose $1$-skeleton is the (undirected) Hasse diagram of the $\s$-Tamari lattice.

Ceballos, Padrol and Sarmiento~\cite{CPS19} realized the Hasse diagram of the $\s$-Tamari lattice as the edge graph of a polyhedral complex which is dual to a subdivision of a subpolytope of a product of simplices and to a fine mixed subdivision of a generalized permutahedron. 

A different realization of the $\s$-Tamari lattice was given by Bell, Gonz\'alez D'Le\'on, Mayorga Cetina and Yip~\cite{vBGDLMCY21} via flow polytopes.

A flow polytope $\fpol[G]$ is the set of valid flows on a directed acyclic graph $G$.
Geometric information about this polytope can be recovered from combinatorial information of the graph, for example computing the volume or constructing certain triangulations.
In the recent literature there has been an increased interest in developing techniques in this direction, see for example \cite{CKM17, JK19, MM15,MM19,MMS19}.
In particular, Danilov, Karzanov and Koshevoy~\cite{DKK12} described a method for obtaining regular unimodular triangulations for $\fpol[G]$ by placing a structure on $G$ called a \defn{framing} (see Section \ref{subsec:flowbackground}). 
Bell et al.~\cite{vBGDLMCY21} used the method of Danilov, Karzanov and Koshevoy to show that the flow polytope on the \defn{$\s$-caracol graph $\car(\s)$} (see the left side of Figure~\ref{fig:s-car_and_oru}) has a framed triangulation whose dual graph is the $\s$-Tamari lattice. 

In a similar vein, we introduce the graph \defn{$\Gs$} which we call the \defn{$\s$-oruga graph} (see Definition \ref{def:Gs} and the right side of Figure \ref{fig:s-car_and_oru}).
Its associated flow polytope possesses a framed triangulation which allows us to answer Conjecture \ref{conj:s-permutahedron} as follows.

\begin{theorem}[{\bf Geometric realizations}]\label{thm:all_realizations}
Let $\s=(s_1,\dots, s_n)$ 
 be a composition. The face poset of the (combinatorial) $\s$-permutahedron $\spermcombi$ is isomorphic to
    \begin{enumerate}
        \item {\bf (Theorem \ref{thm:bij_interiorfacesDKK_facessperm})} The dual of the poset of internal faces of a framed triangulation of a flow polytope $\mathcal{F}_{\Gs}$ of dimension ${\sum_{i=1}^ns_i}$.
        \item {\bf (Theorem \ref{thm:bij_mixed_subdiv})}  The dual of the poset of internal faces of a mixed subdivision of a sum of cubes in $\RR^n$.

        \item {\bf (Theorem \ref{thm:bij_trop_arr})}
        The poset of bounded faces of a polyhedral complex induced by an arrangement of tropical hypersurfaces in $\RR^{n}$. The support of the polyhedral complex is combinatorially equivalent to the $(n-1)$-dimensional permutahedron.
    \end{enumerate}
\end{theorem}

\begin{figure}[ht!]
    \centering
    \tikzset{every picture/.style={line width=0.75pt}} 

\begin{tikzpicture}[x=0.75pt,y=0.75pt,yscale=-.75,xscale=.75]

\draw    (38,148) -- (83.8,148) ;
\draw [shift={(83.8,148)}, rotate = 0] [color={rgb, 255:red, 0; green, 0; blue, 0 }  ][fill={rgb, 255:red, 0; green, 0; blue, 0 }  ][line width=0.75]      (0, 0) circle [x radius= 3.35, y radius= 3.35]   ;
\draw [shift={(38,148)}, rotate = 0] [color={rgb, 255:red, 0; green, 0; blue, 0 }  ][fill={rgb, 255:red, 0; green, 0; blue, 0 }  ][line width=0.75]      (0, 0) circle [x radius= 3.35, y radius= 3.35]   ;
\draw    (83.8,148) -- (129.6,148) ;
\draw [shift={(129.6,148)}, rotate = 0] [color={rgb, 255:red, 0; green, 0; blue, 0 }  ][fill={rgb, 255:red, 0; green, 0; blue, 0 }  ][line width=0.75]      (0, 0) circle [x radius= 3.35, y radius= 3.35]   ;
\draw    (129.6,148) -- (175.4,148) ;
\draw [shift={(175.4,148)}, rotate = 0] [color={rgb, 255:red, 0; green, 0; blue, 0 }  ][fill={rgb, 255:red, 0; green, 0; blue, 0 }  ][line width=0.75]      (0, 0) circle [x radius= 3.35, y radius= 3.35]   ;
\draw  [dash pattern={on 0.84pt off 2.51pt}]  (175.4,148) -- (221.2,148) ;
\draw [shift={(221.2,148)}, rotate = 0] [color={rgb, 255:red, 0; green, 0; blue, 0 }  ][fill={rgb, 255:red, 0; green, 0; blue, 0 }  ][line width=0.75]      (0, 0) circle [x radius= 3.35, y radius= 3.35]   ;
\draw [color={rgb, 255:red, 0; green, 0; blue, 255 }  ,draw opacity=1 ]   (38,148) .. controls (46.8,111) and (124.8,115) .. (129.6,148) ;
\draw [color={rgb, 255:red, 0; green, 0; blue, 255 }  ,draw opacity=1 ]   (38,148) .. controls (51.8,54) and (161.8,56) .. (175.4,148) ;
\draw [color={rgb, 255:red, 0; green, 0; blue, 255 }  ,draw opacity=1 ]   (38,148) .. controls (38.8,19) and (216.8,36) .. (221.2,148) ;
\draw    (221.2,148) -- (267,148) ;
\draw [shift={(267,148)}, rotate = 0] [color={rgb, 255:red, 0; green, 0; blue, 0 }  ][fill={rgb, 255:red, 0; green, 0; blue, 0 }  ][line width=0.75]      (0, 0) circle [x radius= 3.35, y radius= 3.35]   ;
\draw    (83.8,148) .. controls (91.8,220) and (263.8,239) .. (267,148) ;
\draw    (129.6,148) .. controls (145.8,207) and (262.8,205) .. (267,148) ;
\draw    (175.4,148) .. controls (181.8,187) and (266.8,191) .. (267,148) ;
\draw [color={rgb, 255:red, 0; green, 0; blue, 255 }  ,draw opacity=1 ]   (38,148) .. controls (45.8,105) and (117.8,104) .. (129.6,148) ;
\draw [color={rgb, 255:red, 0; green, 0; blue, 255 }  ,draw opacity=1 ]   (38,148) .. controls (44.8,92) and (121.8,101) .. (129.6,148) ;
\draw [color={rgb, 255:red, 0; green, 0; blue, 255 }  ,draw opacity=1 ]   (38,148) .. controls (56.8,66) and (147.8,63) .. (175.4,148) ;
\draw [color={rgb, 255:red, 0; green, 0; blue, 255 }  ,draw opacity=1 ]   (38,148) .. controls (50.8,37) and (174.8,59) .. (175.4,148) ;
\draw [color={rgb, 255:red, 0; green, 0; blue, 255 }  ,draw opacity=1 ]   (38,148) .. controls (38.8,-1) and (231.8,33) .. (221.2,148) ;
\draw [color={rgb, 255:red, 0; green, 0; blue, 255 }  ,draw opacity=1 ]   (38,148) .. controls (30.8,-6) and (234.8,16) .. (221.2,148) ;
\draw  [dash pattern={on 0.84pt off 2.51pt}]  (172,99) -- (195.8,99) ;
\draw    (377,148) -- (422.8,148) ;
\draw [shift={(422.8,148)}, rotate = 0] [color={rgb, 255:red, 0; green, 0; blue, 0 }  ][fill={rgb, 255:red, 0; green, 0; blue, 0 }  ][line width=0.75]      (0, 0) circle [x radius= 3.35, y radius= 3.35]   ;
\draw [shift={(377,148)}, rotate = 0] [color={rgb, 255:red, 0; green, 0; blue, 0 }  ][fill={rgb, 255:red, 0; green, 0; blue, 0 }  ][line width=0.75]      (0, 0) circle [x radius= 3.35, y radius= 3.35]   ;
\draw    (422.8,148) -- (468.6,148) ;
\draw [shift={(468.6,148)}, rotate = 0] [color={rgb, 255:red, 0; green, 0; blue, 0 }  ][fill={rgb, 255:red, 0; green, 0; blue, 0 }  ][line width=0.75]      (0, 0) circle [x radius= 3.35, y radius= 3.35]   ;
\draw    (468.6,148) -- (514.4,148) ;
\draw [shift={(514.4,148)}, rotate = 0] [color={rgb, 255:red, 0; green, 0; blue, 0 }  ][fill={rgb, 255:red, 0; green, 0; blue, 0 }  ][line width=0.75]      (0, 0) circle [x radius= 3.35, y radius= 3.35]   ;
\draw  [dash pattern={on 0.84pt off 2.51pt}]  (514.4,148) -- (560.2,148) ;
\draw [shift={(560.2,148)}, rotate = 0] [color={rgb, 255:red, 0; green, 0; blue, 0 }  ][fill={rgb, 255:red, 0; green, 0; blue, 0 }  ][line width=0.75]      (0, 0) circle [x radius= 3.35, y radius= 3.35]   ;
\draw [color={rgb, 255:red, 0; green, 0; blue, 255 }  ,draw opacity=1 ]   (377,148) .. controls (385.8,111) and (463.8,115) .. (468.6,148) ;
\draw [color={rgb, 255:red, 0; green, 0; blue, 255 }  ,draw opacity=1 ]   (377,148) .. controls (390.8,54) and (500.8,56) .. (514.4,148) ;
\draw [color={rgb, 255:red, 0; green, 0; blue, 255 }  ,draw opacity=1 ]   (377,148) .. controls (377.8,19) and (555.8,36) .. (560.2,148) ;
\draw    (560.2,148) -- (606,148) ;
\draw [shift={(606,148)}, rotate = 0] [color={rgb, 255:red, 0; green, 0; blue, 0 }  ][fill={rgb, 255:red, 0; green, 0; blue, 0 }  ][line width=0.75]      (0, 0) circle [x radius= 3.35, y radius= 3.35]   ;
\draw    (422.8,148) .. controls (437.8,124) and (460.8,133) .. (468.6,148) ;
\draw [color={rgb, 255:red, 0; green, 0; blue, 255 }  ,draw opacity=1 ]   (377,148) .. controls (384.8,105) and (456.8,104) .. (468.6,148) ;
\draw [color={rgb, 255:red, 0; green, 0; blue, 255 }  ,draw opacity=1 ]   (377,148) .. controls (383.8,92) and (460.8,101) .. (468.6,148) ;
\draw [color={rgb, 255:red, 0; green, 0; blue, 255 }  ,draw opacity=1 ]   (377,148) .. controls (395.8,66) and (486.8,63) .. (514.4,148) ;
\draw [color={rgb, 255:red, 0; green, 0; blue, 255 }  ,draw opacity=1 ]   (377,148) .. controls (389.8,37) and (513.8,59) .. (514.4,148) ;
\draw [color={rgb, 255:red, 0; green, 0; blue, 255 }  ,draw opacity=1 ]   (377,148) .. controls (377.8,-1) and (570.8,33) .. (560.2,148) ;
\draw [color={rgb, 255:red, 0; green, 0; blue, 255 }  ,draw opacity=1 ]   (377,148) .. controls (369.8,-6) and (573.8,16) .. (560.2,148) ;
\draw  [dash pattern={on 0.84pt off 2.51pt}]  (511,99) -- (534.8,99) ;
\draw    (468.6,148) .. controls (483.6,124) and (506.6,133) .. (514.4,148) ;
\draw  [dash pattern={on 0.84pt off 2.51pt}]  (514.4,148) .. controls (529.4,124) and (552.4,133) .. (560.2,148) ;
\draw    (560.2,148) .. controls (575.2,124) and (598.2,133) .. (606,148) ;
\draw    (221.2,148) .. controls (231.8,170) and (256.8,170) .. (267,148) ;

\draw (130.91,11.4) node [anchor=north west][inner sep=0.75pt]    {$\mathrm{car}( s)$};
\draw (468.41,9.4) node [anchor=north west][inner sep=0.75pt]    {$\mathrm{oru}( s)$};
\draw (95,95.4) node [anchor=north west][inner sep=0.75pt]  [font=\footnotesize]  {$\textcolor[rgb]{0,0,1}{s}\textcolor[rgb]{0,0,1}{_{n}}$};
\draw (131,62.4) node [anchor=north west][inner sep=0.75pt]  [font=\footnotesize]  {$\textcolor[rgb]{0,0,1}{s}\textcolor[rgb]{0,0,1}{_{n-1}}$};
\draw (174,35.4) node [anchor=north west][inner sep=0.75pt]  [font=\footnotesize]  {$\textcolor[rgb]{0,0,1}{s}\textcolor[rgb]{0,0,1}{_{2}}$};
\draw (430,94.4) node [anchor=north west][inner sep=0.75pt]  [font=\footnotesize]  {$\textcolor[rgb]{0,0,1}{s}\textcolor[rgb]{0,0,1}{_{n} -1}$};
\draw (464,64.4) node [anchor=north west][inner sep=0.75pt]  [font=\footnotesize]  {$\textcolor[rgb]{0,0,1}{s}\textcolor[rgb]{0,0,1}{_{n-1} -1}$};
\draw (518,41.4) node [anchor=north west][inner sep=0.75pt]  [font=\footnotesize]  {$\textcolor[rgb]{0,0,1}{s}\textcolor[rgb]{0,0,1}{_{2} -1}$};

\end{tikzpicture}
    \caption{The $\s$-caracol and $\s$-oruga graphs.}
    \label{fig:s-car_and_oru}
\end{figure}
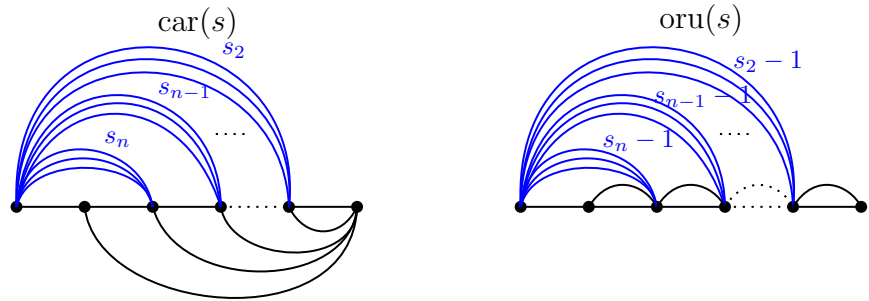

\noindent{\bf Structure of the paper.}
In Section \ref{sec:combinatorics_of_the_s_permutahedron_stirling_permutations} we translate the combinatorics of $\s$-decreasing trees, and the definitions of the $\s$-weak order and the $\s$-permutahedron to the language of Stirling $\s$-permutations which we will use extensively throughout this paper. 
In Section~\ref{sec:subdivisions_flow_polytopes} we provide the necessary background on flow polytopes which we will need for our geometric constructions and present our first geometric realization of the $\s$-permutahedron as the dual of a framed triangulation of the flow polytope $\fpol[\Gs]$. 
This first realization provides a formula for the $h^*$-polynomial of $\fpol[\Gs]$ by enumerating Stirling $\s$-permutations with respect to descents.
In Section \ref{sec:Cayley_trick} we provide the background on the Cayley trick and mixed subdivisions, and present our second geometric realization of the $\s$-permutahedron as the dual of a fine mixed subdivision of a sum of hypercubes. 
In Section \ref{sec:realization_polyhedral_complex} we provide background on tropical dualization in the Cayley trick setting, and present our third realization of the $\s$-permutahedron, as the collection of bounded faces of an arrangement of tropical hypersurfaces. This third realization provides a complete answer to Ceballos and Pons conjecture in the case where $\s$ is a composition. Examples of this realization are available on this \href{https://sites.google.com/view/danieltamayo22/gallery-of-s-permutahedra}{webpage}\footnote{{https://sites.google.com/view/danieltamayo22/gallery-of-s-permutahedra}} and code can be found on this \href{https://cocalc.com/ahmorales/s-permutahedron-flows/demo-realizations}{webpage}\footnote{{https://cocalc.com/ahmorales/s-permutahedron-flows/demo-realizations}}.
In Section \ref{sec:lidskii_main} we present formulas enumerating $\s$-decreasing trees and Stirling $\s$-permutations. These formulas follow from applications of the Lidskii formulas. 

We leave an open question on realizing those formulas as a geometric Lidskii-type decomposition of  $\mathcal{F}_{\Gs}$.
Finally, in Section~\ref{sec:further} we discuss ongoing work and some ramifications of our results.

\section{Combinatorics of the \texorpdfstring{$\s$}{s}-permutahedron in the language of Stirling \texorpdfstring{$\s$}{s}-permutations}\label{sec:combinatorics_of_the_s_permutahedron_stirling_permutations}

The $\s$-weak order can also be described by Stirling $\s$-permutations, which we now define.
Throughout the remainder of this article, unless specified, $\s$ is assumed to be a composition (i.e.\ a vector with positive integer entries).  This restriction is necessary for us to connect the combinatorics of $\s$-decreasing trees to the underlying geometry.

Let $\s=(s_1,\ldots, s_n)$ be a composition. 
A \defn{Stirling $\s$-permutation} is a permutation of the word $1^{s_1}2^{s_2}\ldots n^{s_n}$ that avoids the pattern $121$, which means that there is never a letter $j$ in between two occurrences of $i$ with $i<j$.
We denote by $\setperms$ the set of all Stirling $\s$-permutations.

The set of Stirling $\s$-permutations is in bijection with the set of $\s$-decreasing trees.
This bijection is obtained by reading nodes along the in-order traversal of the \defn{caverns} (spaces between 
 consecutive siblings) of an $\s$-decreasing trees (see Figure~\ref{fig:2112-decreasing-inordered-tree}).
Note that this bijection induces a correspondence between the prefixes of a Stirling $\s$-permutation $w$ and the leaves of its corresponding tree $T(w)$.

\begin{figure}[ht]
    \centering\includegraphics[scale=1.5]{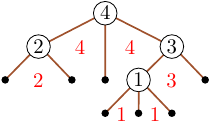}
    \caption{A $(2,1,1,2)$-decreasing tree with vertices labeled via in-order. The corresponding Stirling $\s$-permutation is $w=244113$.}
    \label{fig:2112-decreasing-inordered-tree}
\end{figure}

Analogous to the case of classical permutations, the cover relation in the $\s$-weak order can be described in terms of transpositions of substrings in Stirling $\s$-permutations. 

Let $w$ be a Stirling $\s$-permutation.
For $a\in [n]$, we define the \defn{$a$-block $B_a$} of $w$ to be the shortest substring of $w$ containing all $s_a$ occurrences of $a$. 
In Example~\ref{eg:s-perm}, we see that the $5$-block of $w=33725455716$ is $B_5=5455$.
Note that an $a$-block of $w$ necessarily starts and ends with $a$ by minimality, and contains only letters in $[a]$ because $w$ is $121$-avoiding. 
Furthermore for $a<c$, $w$ contains the consecutive substring $ac$ if and only if it is of the form $w=u_1B_acu_2$, where $u_1$ and $u_2$ denote consecutive substrings of $w$.

Let $w$ be a Stirling $\s$-permutation.
A pair $(a,c)$ with $1\leq a < c \leq n$ is called an \defn{ascent} of $w$ if $ac$ is a consecutive substring of $w$. It is a \defn{descent} of $w$ if $ca$ is a substring of $w$. 
If $w$ is of the form $w=u_1B_acu_2$ and $a<c$, the \defn{transposition} of $w$ along the ascent $(a,c)$ is the Stirling $\s$-permutation $u_1cB_au_2$. 
We denote by \defn{$\inv(w)$} the multiset of inversions formed by pairs $(c,a)$ 
with multiplicity \defn{$\#_w(c,a)$}$\in [0, s_c]$ the number of occurrences of $c$ that precede the $a$-block in $w$. 
As in the case of tree-rotations, if $A$ is a subset of ascents of $w$, we denote by \defn{$w+A$} the Stirling $\s$-permutation whose inversion set is the transitive closure of $\inv(w)+A$. We have the following correspondence between concepts on the family of $s$-decreasing trees and on the family of Stirling $s$-permutations, whose proof follows easily from the definitions.

\begin{lemma}\label{lem:ascdesc}
Let $w$ be a Stirling $\s$-permutation and $T(w)$ its corresponding $\s$-decreasing tree. 
Let $1\leq a<c\leq n$. 
\begin{enumerate}
\item[(a)] The pair $(a,c)$ is an ascent of $T(w)$ if and only if it is an ascent of $w$.
\item[(b)] The pair $(a,c)$ is a descent of $T(w)$ if and only if it is a descent of $w$.
\item[(c)] $\#_{T(w)}(c,a)=\#_w(c,a)$.
\end{enumerate}
Moreover, suppose $(a,c)$ is an ascent of $T=T(w)$ so that $w$ is the of the form $w=u_1B_acu_2$.
Then $T'$ is the $\s$-tree rotation of $T$ along $(a,c)$ if and only if $T'=T(w')$ where $w'=u_1 cB_au_2$.
\end{lemma}

\begin{corollary}\label{prop:cover_relations_multiperm}
Let $w$ and $w'$ be Stirling $\s$-permutations.  
Then $w'$ covers $w$ in the $\s$-weak order if and only if $w'$ is the transposition of $w$ along an ascent.
\end{corollary}

\begin{example}\label{eg:s-perm}
Let $\s=(1,1,2,1,3,1,2)$ and consider the $\s$-permutation $w=33725455716$. 
The transposition of $w$ along the ascent $(5, 7)$ switches the $5$-block of $w$ with the $7$ that immediately follows it and yields $w'=3372\textcolor{red}{7}\textcolor{blue}{5455}16$.
The corresponding $\s$-decreasing tree $T=T(w)$ is shown on the left of Figure~\ref{fig:tree_rotation}.  
The rotation of $T$ along the ascent $(5,7)$ yields $T'=T(w')$.
\end{example}

\begin{figure}[ht!]
\centering
\begin{tikzpicture}
\begin{scope}[scale=0.35, xscale=2.0, yscale=2.5, xshift=0, yshift=0]
\node at (-4,0) {$T:$};
	\treenode[](n7) at (0,0) {\scriptsize$7$};
	\treenode[](n6) at (2,-1) {\scriptsize$6$};
	\treenode[](n5) at (0,-1) {\scriptsize$5$};
	\treenode[](n4) at (-.5,-2) {\scriptsize$4$};
	\treenode[](n3) at (-2,-1) {\scriptsize$3$};
	\treenode[](n2) at (-1.5,-2) {\scriptsize$2$};
	\treenode[](n1) at (2.5,-2) {\scriptsize$1$};
	\leafnode[](l1) at (-2.5,-2) {};
	\leafnode[](l2) at (-3.5,-2) {};
	\leafnode[](l3) at (-4.5,-2) {};
	\leafnode[](l4) at (-2.5,-3) {};
	\leafnode[](l5) at (-1.5,-3) {};
	\leafnode[](l6) at (-.5,-3) {};
	\leafnode[](l7) at (.5,-3) {};	
	\leafnode[](l8) at (.5,-2) {};
	\leafnode[](l9) at (1.5,-2) {};	
	\leafnode[](l10) at (2.5,-3) {};
	\leafnode[](l11) at (3.5,-3) {};
	\leafnode[](l12) at (3.5,-2) {};

	\draw[very thick, color=sienna] (n7)--(n5);
	\draw[very thick, color=sienna] (n7)--(n3);
	\draw[very thick, color=sienna] (n7)--(n6);
	\draw[very thick, color=sienna] (n3)--(l1);
	\draw[very thick, color=sienna] (n3)--(l2);
	\draw[very thick, color=sienna] (n3)--(l3);
	\draw[very thick, color=sienna] (n5)--(n2);
	\draw[very thick, color=sienna] (n5)--(n4);
	\draw[very thick, color=sienna] (n5)--(l8);
	\draw[very thick, color=sienna] (n5)--(l9);
	\draw[very thick, color=sienna] (n2)--(l4);
	\draw[very thick, color=sienna] (n2)--(l5);
	\draw[very thick, color=sienna] (n4)--(l6);
	\draw[very thick, color=sienna] (n4)--(l7);
	\draw[very thick, color=sienna] (n1)--(l10);
	\draw[very thick, color=sienna] (n1)--(l11);
	\draw[very thick, color=sienna] (n6)--(n1);
	\draw[very thick, color=sienna] (n6)--(l12);
\end{scope}
\begin{scope}[scale=0.35, xscale=2.0, yscale=2.5, xshift=280, yshift=0]
\node at (-3.5,0) {$T':$};
	\treenode[](n7) at (0,0) {\scriptsize$7$};
	\treenode[](n6) at (2,-1) {\scriptsize$6$};
	\treenode[](n5) at (2,-2) {\scriptsize$5$};
	\treenode[](n4) at (1.5,-3) {\scriptsize$4$};
	\treenode[](n3) at (-2,-1) {\scriptsize$3$};
	\treenode[](n2) at (0,-1) {\scriptsize$2$};
	\treenode[](n1) at (3.5,-3) {\scriptsize$1$};
	\leafnode[](l1) at (-2,-2) {};
	\leafnode[](l2) at (-3,-2) {};
	\leafnode[](l3) at (-4,-2) {};
	\leafnode[](l4) at (-1,-2) {};
	\leafnode[](l5) at (0,-2) {};
	\leafnode[](l6) at (1.5,-4) {};
	\leafnode[](l7) at (2.5,-4) {};	
	\leafnode[](l8) at (0.5,-3) {};
	\leafnode[](l9) at (2.5,-3) {};	
	\leafnode[](l10) at (3.5,-4) {};
	\leafnode[](l11) at (4.5,-4) {};
	\leafnode[](l12) at (3,-2) {};
	\draw[very thick, color=sienna] (n7)--(n3);
	\draw[very thick, color=sienna] (n7)--(n2);
	\draw[very thick, color=sienna] (n7)--(n6);
	\draw[very thick, color=sienna] (n3)--(l1);
	\draw[very thick, color=sienna] (n3)--(l2);
	\draw[very thick, color=sienna] (n3)--(l3);
	\draw[very thick, color=sienna] (n5)--(n4);
	\draw[very thick, color=sienna] (n5)--(n1);
	\draw[very thick, color=sienna] (n5)--(l8);
	\draw[very thick, color=sienna] (n5)--(l9);
	\draw[very thick, color=sienna] (n2)--(l4);
	\draw[very thick, color=sienna] (n2)--(l5);
	\draw[very thick, color=sienna] (n4)--(l6);
	\draw[very thick, color=sienna] (n4)--(l7);
	\draw[very thick, color=sienna] (n1)--(l10);
	\draw[very thick, color=sienna] (n1)--(l11);
	\draw[very thick, color=sienna] (n6)--(n5);
	\draw[very thick, color=sienna] (n6)--(l12);
\end{scope}
\end{tikzpicture}
\caption{A rotation of the $\s$-decreasing tree $T$ along the ascent $(5,7)$ yields $T'$.  $T$ corresponds with $w=33725455716$ and $T'$ corresponds with $w'=33727545516$, which is the transposition of the $5$-block of $w$ with the $7$ that immediately follows it.
}
\label{fig:tree_rotation}
\end{figure}
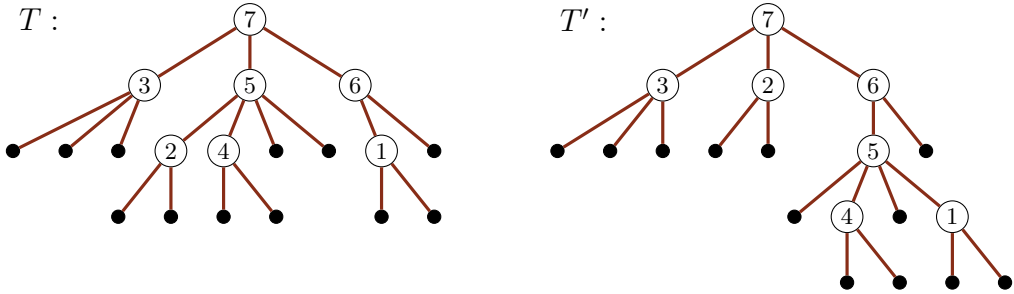

\begin{remark}\label{def:spermcombi_withperm}
    The $\s$-permutahedron $\spermcombi$ of Definition~\ref{def:s-permutahedron} can be alternatively defined as the combinatorial complex with faces $(w,A)$ where $w$ is a Stirling $\s$-permutation and $A$ is a subset of ascents of $w$.
\end{remark}

\section{Subdivisions of flow polytopes}\label{sec:subdivisions_flow_polytopes}

\subsection{Background on DKK triangulations of flow polytopes}\label{subsec:flowbackground}

Let $G=(V, E)$ be a loopless connected oriented graph on vertices $V=\{v_0, \ldots, v_n\}$ with edges oriented from $v_i$ to $v_j$ if $i<j$ and such that $v_0$ (respectively $v_n$)  is the only source (respectively sink) of $G$. A vertex is an \defn{inner vertex} if it is not a source and not a sink. We allow multiple edges between pairs of vertices. 
For any vertex $v_i$ we denote by \defn{$\cI_i$} its set of incoming edges, and by \defn{$\cO_i$} its set of outgoing edges.

Given a vector $\mathbf{a}=(a_0, a_1 \ldots, a_{n-1}, a_n)$ such that $\sum_i a_i=0$, a \defn{flow} of $G$ with \defn{netflow} $\mathbf{a}$ is a vector $(f_e)_{e\in E}\in (\RR_{\geq 0})^E$ such that: $\sum_{e\in \cI_i}f_e + a_i=\sum_{e\in \cO_i}f_e$ for all $i\in [0, n]$. 
A flow $(f_e)_{e\in E}$ of $G$ is called an \defn{integer flow} if all $f_e$ are integers. 
We denote by \defn{$\mathcal{F}_G^{\mathbb{Z}}({\bf a})$} the set of integer flows of $G$ with netflow $\mathbf{a}$. 
The \defn{flow polytope} of $G$ is
$$\fpol({\bf a}) = \Big\{ (f_e)_{e\in E} \text{ flow of } G \text{ with netflow } {\bf a} \Big\} \subseteq \RR^E.$$

When the netflow is not specified, i.e.\ when we write $\fpol[G]$, it is assumed to be $\mathbf{a}=(1, 0, \ldots, 0, -1)$. In this case, $\fpol$ is a polytope of dimension $|E|-|V|+1$, and the vertices of $\fpol$ correspond exactly to indicator vectors of the routes of $G$ (\cite[Corollary  3.1]{GS78}), where a \defn{route} of $G$ is a path from $v_0$ to $v_n$ i.e.\ a sequence of edges $((v_0, v_{k_1}), (v_{k_1}, v_{k_2}), \ldots, (v_{k_l}, v_n))$, with $0<k_1<k_2< \ldots <k_l<n$.

Flow polytopes admit several nice subdivisions that can be understood via combinatorial properties of the graph $G$, in particular the triangulations defined by Danilov, Karzanov and Koshevoy in~\cite{DKK12}, that will be our main tool for obtaining geometric realizations of $\s$-permutahedra.

Let $P$ be a route of $G$ that contains vertices $v_i$ and $v_j$.
We denote by \defn{$Pv_i$} the prefix of $P$ that ends at $v_i$, \defn{$v_iP$} the suffix of $P$ that starts at $v_i$ and \defn{$v_iPv_j$} the subroute of $P$ that starts at $v_i$ and ends at $v_j$. 

A \defn{framing $\preceq$} of $G$ is a choice of linear orders $\preceq_{\cI_i}$ and $\preceq_{\cO_i}$ on the sets of incoming and outgoing edges for each inner vertex $v_i$. 
This induces a total order on the set of partial routes from $v_0$ to $v_i$ (respectively from $v_i$ to $v_n$) by taking $Pv_i\preceq Qv_i$ if $e_P \preceq_{\cI_j} e_Q$ where $v_j$ is the first vertex after which the two partial routes coincide, and $e_P$, $e_Q$ are the edges of $P$ and $Q$ that end at $v_j$ (respectively $v_iP\preceq v_iQ$ if $e_P \preceq_{\cO_j} e_Q$ where $v_j$ is the last vertex before which the two partial routes coincide, and $e_P$, $e_Q$ are the edges of $P$ and $Q$ that start at $v_j$).
When $G$ is endowed with such a framing $\preceq$, we say that $G$ is \defn{framed}. See Figure~\ref{fig:framed_graph} for an example.

Let $P$ and $Q$ be routes of $G$ with a common subroute between inner vertices $v_i$ and $v_j$ (possibly with $v_i=v_j$). 
We say that $P$ and $Q$ are \defn{in conflict} at $[v_i,v_j]$ if the initial parts $Pv_i$ and $Qv_i$ are ordered differently than the final parts $v_jP, v_jQ$.
Otherwise we say that $P$ and $Q$ are \defn{coherent} at $[v_i,v_j]$. 
We say that $P$ and $Q$ are \defn{coherent} if they are coherent at each common inner subroute. See Figure~\ref{fig:coherent routes}.

\begin{figure}[ht]
    \centering
    \includegraphics{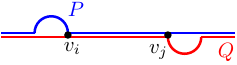} \qquad \includegraphics{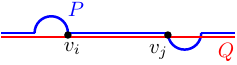}
    \caption{Illustration of routes $P$ and $Q$ that are coherent (on the left) and in conflict (on the right) with respect to the framing where the incoming/outgoing edges at each vertex are ordered from top to bottom.}
    \label{fig:coherent routes}
\end{figure}

The relation on routes being coherent is reflexive and symmetric and we can consider sets of mutually coherent routes which are called the  \defn{cliques} of $(G, \preceq)$. We denote  by \defn{$\cliques$} the set of cliques of $(G, \preceq)$, and \defn{$\maxcliques$} the set of maximal collection of cliques under inclusion. 
For a set of routes $C$, let $\Delta_C$ be the convex hull of the vertices of $\fpol$ corresponding to the elements in $C$.

\begin{theorem}[{\cite[Theorem 1 \& 2]{DKK12}}]\label{thm:DKKsimplices}
The simplices $\{\Delta_C\, |\, C \in \maxcliques\}$ are the maximal cells of a regular triangulation of $\fpol$.
\end{theorem}

\begin{proof}
    The formulation in \cite{DKK12} is in terms of the cone $\cF_{+}$ of flows with any netflow $\mathbf{a} = (\lambda, 0, \ldots, 0, -\lambda)$, for $\lambda\in \RR_{+}$. We only need to intersect this cone with the affine hyperplane corresponding to taking $\lambda=1$ to obtain the theorem in our formulation for the flow polytope $\fpol$.
\end{proof}

The triangulation obtained this way is the \defn{DKK triangulation} of $\fpol$ with respect to the framing $\preceq$ and we denote it by \defn{$\triangDKK$}.
The word \emph{regular} in the above theorem is related to the existence of an admissible height function, see the definition in Section \ref{sec:background_tropicalization} and Section~\ref{subsec:height_functions}.

Postnikov~\cite{P1014} and Stanley~\cite{S00} used a recursive {\em subdivision lemma} (see \cite[Proposition 4.1]{MM15}) to show that the volume of $\fpol$ equals the number of integer flows in $\mathcal{F}_G^{\mathbb{Z}}({\bf d})$, with ${\bf d}=(0,d_1,\ldots,d_{n-1},-\sum_i d_i)$ and $d_i=\indeg_G(v_i)-1$ (see Corollary~\ref{cor:volume 10...0-1 case}). 
This recursive subdivision can be made compatible with DKK triangulations in what are called \defn{framed Postnikov--Stanley triangulations} \cite[Section 7.1]{MMS19}.
As explained by M\'esz\'aros, Morales, and Striker~\cite[Definition 7.5]{MMS19}, this leads to the following explicit bijection between the maximal cliques of $(G, \preceq)$ and the integer flows on $G$ with netflow $\bf d$.

Let $(G,\preceq)$ be a framed graph with netflow ${\bf d}=(0,d_1,\ldots,d_{n-1},-\sum_i d_i)$ where $d_i=\indeg_G(v_i)-1$. 
We define the function 
\[\Omega_{G,\preceq}: 
\maxcliques \to \mathcal{F}_G^{\mathbb{Z}}({\bf d}): C \mapsto (n_C(e)-1)_{e\in E(G)},\]
where $n_C(e)$ is the number of times the edge $e=(v_i, v_j)$ appears in the set of prefixes $\{Pv_j \mid P \in C\}$ of the maximal clique.

\begin{theorem}[{\cite[Theorem 7.8]{MMS19}}]
\label{thm:bij_cliques_intflows}
Given a framed graph $(G,\preceq)$, the map $\Omega_{G,\preceq}$ is a bijection between maximal cliques in $\maxcliques$ and integer flows in $\mathcal{F}^{\mathbb{Z}}_G({\bf d})$.

\end{theorem}

As a corollary, we obtain the following result of Postnikov--Stanley (unpublished, see \cite[Sec. 6]{MM15} and \cite[Sec. 3]{MS} for proofs) and Baldoni--Vergne \cite[Thm. 38]{BVkpf}.

\begin{corollary}[{\cite{S00} and \cite[Thm. 38]{BVkpf}}] \label{cor:volume 10...0-1 case}
For a graph $G$ on $\{v_0,\ldots,v_{n}\}$ with netflow ${\bf d}=(0,d_1,\ldots,d_{n-1},-\sum_i d_i)$ where $d_i=\indeg(v_i)-1$, we have that 
\[
\vol \fpol(1,0,\ldots,0,-1) = \#\mathcal{F}^{\mathbb{Z}}_G({\bf d}),
\]
where $\vol$ denotes the normalized volume of a polytope.
\end{corollary}

\subsection{The flow polytope realization }\label{subsec:realization_flowpolytopes}
We introduce the $\s$-oruga graphs along with a fixed framing, and apply the combinatorial method of Danilov, Karzanov and Koshevoy to obtain a triangulation of the associated flow polytope. 
Combining previous results, we will have bijections between $\s$-decreasing trees $\mathcal{T}_{\s}$, Stirling $\s$-permutations $\setperms$, integer $\mathbf{d}$-flows on $\Gs$, and maximal cliques in the framed graph $(\Gs, \preceq)$, represented in the following diagram:

\begin{center}
\begin{tikzpicture}
\begin{scope}[xscale=2]
\node[](n1) at (0,0) {$\settrees$};
\node[](n2) at (1.4,0) {$\setperms$};
\node[](n3) at (3,0) {$\mathcal{F}_{\Gs}^{\mathbb{Z}}(\mathbf{d})$};
\node[](n4) at (5.5,0) {$\maxcliques[\Gs]$};

\draw[-stealth] (n1)--(n2);
\draw[-stealth] (n2)--(n1);
\draw[-stealth] (n2)--(n3);
\draw[-stealth] (n3)--(n2);
\draw[-stealth] (n3)--(n4);
\draw[-stealth] (n4)--(n3);
\draw[-stealth] (n1) to[out=-30,in=210] (n3);
\draw[-stealth] (n3) to[out=210,in=-30] (n1);
\draw[-stealth] (n2) to[out=-30,in=210] (n4);
\draw[-stealth] (n4) to[out=210,in=-30] (n2);

\node[] at (.7,.3) {\tiny Sec~\ref{sec:combinatorics_of_the_s_permutahedron_stirling_permutations}};
\node[] at (2.05,.3) {\tiny Prop~\ref{prop:bij_simplices_permutations}};
\node[] at (3.95,.3) {\tiny Thm~\ref{thm:bij_cliques_intflows}};
\node[] at (1.6,-1) {\tiny Rem~\ref{thm:bij_simplices_trees}};
\node[] at (3.5,-1) {\tiny Lem~\ref{lem:max_clique}};

\end{scope}
\end{tikzpicture}
\end{center}

\subsubsection{The $\s$-oruga graph and a DKK triangulation of its flow polytope}

\begin{definition} \label{def:Gs}
Let $\s=(s_1, \ldots, s_n)$ be a composition, and for convenience of notation we also set $s_{n+1}=2$.
The framed graph \defn{$(\Gs, \preceq)$} consists of vertices $\{v_{-1},v_0, \ldots, v_n\}$ and
\begin{itemize}
    \itemsep0em 
    \item for $i\in [1,n+1]$, there are $s_i-1$ \defn{source-edges} $(v_{-1}, v_{n+1-i})$ labeled $e^i_1$, \ldots,  $e^i_{s_{i}-1}$,
    \item for $i\in [1,n]$, there are two edges $(v_{n+1-i-1}, v_{n+1-i})$ called \defn{bump} and \defn{dip} labeled $e^{i}_0$ and $e^{i}_{s_{i}}$,
    \item the incoming edges of $v_{n+1-i}$ are ordered $e^i_j \prec_{\cI_{n+1-i}} e^i_k$ for $0\leq j < k \leq s_{i}$,
    \item the outgoing edges of $v_{n+1-i}$ are ordered $e^{i-1}_0\prec_{\cO_{n+1-i}} e^{i-1}_{s_{i-1}}$.
\end{itemize}
We call $\Gs$ the \defn{$\s$-oruga graph}. 
We will also denote by \defn{$\oruga$} the induced subgraph of $\Gs$ with vertices $\{v_0,\ldots,v_n\}$ and call this the \defn{oruga graph} of length $n$.

\end{definition}

\begin{figure}[t!]
    \begin{subfigure}[b]{0.5\textwidth}
    \centering
    \raisebox{12pt}{
    \includegraphics{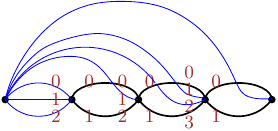} }
    \caption{}
    \label{fig:framed_graph}
    \end{subfigure}
    \begin{subfigure}[b]{0.3\textwidth}
    \centering
\raisebox{10pt}{\includegraphics{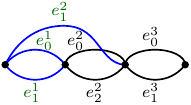}}
    \caption{}
        \label{fig:graph_framing_12}
\end{subfigure}
    \caption{(a) The graph $\Gs$ for $s=(2,3,2,2)$ with framing shown in red. That is, an ordering on incoming edges and outgoing edges at each inner vertex from  ``top to bottom". (b) The graph $\Gs$ for $s=(1,2,1)$ with edge labels.}
\end{figure}

Figure~\ref{fig:framed_graph} and Figure~\ref{fig:graph_framing_12} show examples of this construction. 
In this article, we choose to draw the graph $\Gs$ in such a way that the framing of the incoming and outgoing edges at each inner vertex is ordered from ``top to bottom''.
Note that the corresponding flow polytope $\fpol[\Gs]$ has dimension $\sizes=\sum_{i=1}^n s_i$.

The routes of $\Gs$ will play a key role, thus we will describe them as \defn{$\route(k, t, \delta)$} intuitively as follows. Every route of $\Gs$ starts from $v_{-1}$, lands in a vertex $v_{n+1-k}$ via a source-edge labelled $e_{t}^k$ and then follows $k-1$ edges that are either bumps or dips denoted by a $01$-vector $\delta$. 

Formally, for $k\in[n+1]$, $t\in [1, s_k-1]$, and $\delta=(\delta_1, \ldots, \delta_{k-1})\in \{0, 1\}^{k-1}$, we denote by \defn{$\route(k, t, \delta)$} the sequence of edges $(e^{k}_{t_k}, \, e^{k-1}_{t_{k-1}}, \, \ldots, \, e^1_{t_1})$
where \begin{itemize}
    \itemsep0em 
    \item $t_k=t$,
    \item for all $j\in [1, k-1]$, $t_j =\delta_j s_j$.
\end{itemize}

\begin{figure}[ht!]
\centering
\includegraphics[scale=1.0]{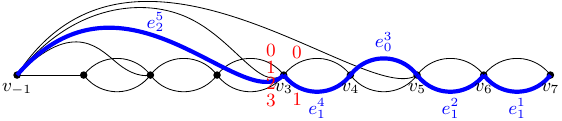}
\caption{The $\s$-oruga graph $\Gs$ for $\s=(1,1,2,1,3,1,2)$.  The framing orders $\cI_3$ and $\cO_3$ are shown at vertex $v_3$.  The route $\route(5,2, (1,0,1,1))$ is shown in bold blue.
}
\label{fig:s-oruga_graph_route2}
\end{figure}

\begin{remark}\label{rem:vertex_v-1}
    Although the graph $\Gs$ starts with vertex $v_{-1}$ instead of $v_0$ all the technology of Section~\ref{subsec:flowbackground} can be applied to it. To see this, we can either contract the edge $e_1^{n+1}$ to obtain a graph whose flow polytope is integrally equivalent to $\fpol$, or simply relabel the vertices with $[0,n+1]$. The resulting graph has flows and routes directly in bijection with the flows and routes of $\Gs$.
\end{remark}

\begin{proposition}
\label{prop:bij_simplices_permutations}
Let $\s$ be a composition and let $\mathbf{d}=(0,0,s_n, s_{n-1},\ldots, s_2, - \sum_{i=2}^n s_i)$. 
The set of Stirling $\s$-permutations is in bijection with the set of integer $\mathbf{d}$-flows of $\Gs$.
\end{proposition}
\begin{proof}
First, we notice that an integer flow $(f_e)_e$ on $\Gs$ with netflow $\mathbf{d}$ necessarily has zero flow on every source-edge, so $(f_e)_e$ is characterized by the fact that the total flow on each pair of bump and dip edges satisfies $f_{e_0^i} + f_{e_{s_i}^{i}} = s_n + \cdots + s_{i+1}$ for all $i\in[1,n-1]$.
Thus to describe an integer $\mathbf{d}$-flow on $\Gs$, it is enough to determine the flow on the bump edges $e_0^i$ for all $i\in [1,n-1]$.
Given a Stirling $\s$-permutation $w$, let $f_{e_0^i}$ be the number of letters strictly greater than $i$ that occur before the $i$-block $B_i$ in $w$.
This implies $0\leq f_{e_0^i} \leq s_n+\cdots+s_{i+1}$, and thus defines an integer $\mathbf{d}$-flow on $\Gs$.

Conversely, any Stirling $\s$-permutation can be built iteratively by an insertion algorithm associated to a choice of integers $f_{e_0^i} \in [0,s_n+\cdots+s_{i+1}]$ for $i\in [1,n-1]$ in the following way.
Start with a block of $s_n$ consecutive copies of $n$ (step $i=0$).
At step $i$ for $i\in [1,n-1]$, there are $s_n+\cdots+s_{n-i+1}+1$ possible positions for the next insertion.  We insert a block of $s_{n-i}$ consecutive copies of $(n-i)$ in the $(f_{e_0^{n-i}})$-th position.  
This creates a $121$-avoiding permutation of the word $1^{s_1}2^{s_2}\cdots n^{s_n}$.
\end{proof}
See Figure~\ref{fig:insertion_algorithm} for an example illustrating the insertion algorithm described in the proof of Proposition~\ref{prop:bij_simplices_permutations}.
\begin{figure}[!ht]
    \centering
    \includegraphics[scale=0.9]{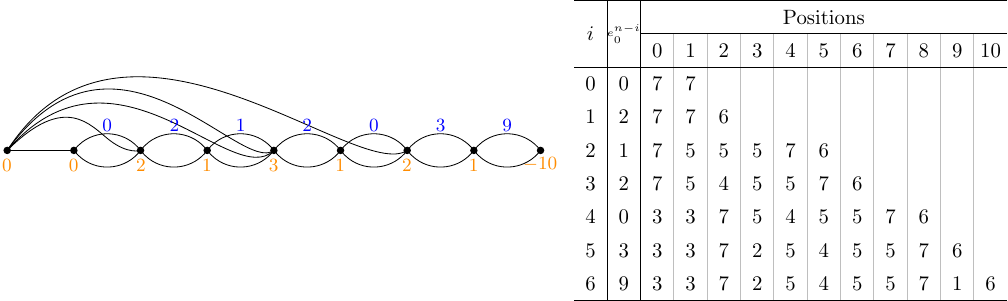}
    \caption{An integer $\mathbf{d}$-flow of $\Gs$ (only the flow on the bump edges is shown) and the steps of the insertion algorithm of Proposition~\ref{prop:bij_simplices_permutations} that output the corresponding Stirling $\s$-permutation $w=33725455716$.
    }
    \label{fig:insertion_algorithm}
\end{figure}

By Corollary~\ref{cor:volume 10...0-1 case}, since the normalized volume of the flow polytope $\mathcal{F}_{\Gs}$ is the number of integer $\mathbf{d}$-flows on $\Gs$, then we obtain the following as an immediate corollary.

\begin{corollary} \label{cor:volume is number of trees}
Given a composition $\s$, then 
\[\vol\fpol[\Gs] 
    = \#\settrees 
    = \#\setperms
    = \prod_{i=1}^{n-1}\left(1+s_{n-i+1}+s_{n-i+2}+\cdots + s_n\right).
\]
\end{corollary}

\begin{remark}\label{thm:bij_simplices_trees}
We can also give an explicit correspondence between $\s$-decreasing trees and integer $\mathbf{d}$-flows of $\Gs$.
Note that this correspondence holds in the more general setting where $\s$ is a weak composition and we consider integer $\mathbf{d}$-flows on the oruga graph~$\oruga$ since the source-edges do not play any role.

Given an integer $\mathbf{d}$-flow $(f_e)_e$ of $\Gs$ (note again that it is enough to know the values $f_{e_0^i}$ for $i\in [1,n-1]$ to determine the entire integer flow), we build an $\s$-decreasing tree inductively as follows.
Start with the tree given by the node $n$ and $s_n+1$ leaves. 
At step $i$ for $i \in [1, n-1]$, we have a partial $\s$-decreasing tree with labelled nodes $n$ to $n+1-i$, and $1+\sum_{k=n+1-i}^{n} s_{k}$ leaves that we momentarily label from $0$ to $\sum_{k=n+1-i}^{n} s_{k}$ along the counterclockwise traversal of the partial tree. 
Attach the next node $n-i$, with $s_{n-i}+1$ pending leaves, to the leaf of the partial tree labeled $f_{e^{n-i}_0}$. 
This procedure produces decreasing trees with the correct number of children at each node. 
Hence, after the $n$-th step we obtain an $\s$-decreasing tree.
Reciprocally, any $\s$-decreasing tree can be built iteratively in this way, so it is associated to a choice of integers $f_{e^i_0}\in [0, \sum_{k=n+1-i}^{n} s_{k}]$ for all $i\in[1, n-1]$. The interested reader can verify that this procedure applied to the flow in the example of Figure \ref{fig:insertion_algorithm} produces the tree $T$ on the left of Figure \ref{fig:tree_rotation}.

\end{remark}

We can now explicitly describe the DKK maximal cliques of coherent routes of $\Gs$ via Stirling $\s$-permutations.
This is an important construction for the results which follow. 

Let $\s$ be a composition, and $u$ a (possibly empty) prefix of a Stirling $\s$-permutation.
For all $a\in [1,n]$, we denote by $t_a$ (or \defn{$t_a(u)$} if $u$ is not clear from the context) the number of occurrences of $a$ in $u$, and we denote by $c$ the smallest value in $[1,n]$ such that $0<t_c<s_{c}$. 
If there is no such value, we set $c=n+1$ and $t_{n+1}=1$. 
The definition of $c$ implies that for all $a<c$, either $t_a=0$ or $t_a=s_a$.
Then we define \defn{$\routep{u}$} to be the route $(e^{c}_{t_c}, \, e^{c-1}_{t_{c-1}}, \, \ldots, \, e^1_{t_1})$.

For example, for the subword $u=3372545$ of $w=33725455716$ in the example of Figure \ref{fig:insertion_algorithm} we have that $c=5$, $t_5=2,t_4=1,t_3=2,t_2=1,t_1=0$ so $\routep{u}=(e^{5}_{2}, \, e^{4}_{1},e^{3}_{2},e^{2}_{1},e^{1}_{0})=\route(5, 2, (1,1,1,0))$.

Let $w$ be a Stirling $\s$-permutation.
For $i\in [\sizes]$, we denote by \defn{$w_i$} the $i$-th letter of $w$, and for $i\in [0, \sizes]$ we denote by \defn{$\prefix{w}{i}$} the prefix of $w$ of length $i$, with $\prefix{w}{0}:=\emptyset$. 
Let \defn{$\Delta_w$} be the set of routes $\{\routep{\prefix{w}{i}}\, |\, i\in [0, \sizes]\}$ and identify it with the simplex whose vertices are the indicator vectors of these routes.

Note that each maximal clique always contains the routes $\routep{\prefix{w}{0}}=(e^{n+1}_1, e^n_0, \ldots, e^1_0)=\route(n+1, 1, (0)^n)$ and $\routep{\prefix{w}{\sizes}}=(e^{n+1}_1, e^n_{s_n}, \ldots, e^1_{s_1})=\route(n+1, 1, (1)^n)$.
See Figure~\ref{fig:clique_3221} for the example of $\Delta_w$  corresponding to the Stirling $(1,2,1)$-permutation $w=3221$.

\begin{figure}[ht]
\begin{subfigure}[ht]{\linewidth}
 \[\routep{\prefix{w}{0}} = {\includegraphics[scale=0.8]{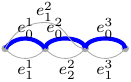}}\]
\end{subfigure}
\begin{subfigure}[ht]{\linewidth}
 \[\routep{\prefix{w}{1}} = {\includegraphics[scale=0.8]{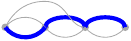}}\]
\end{subfigure}
\begin{subfigure}[ht]{\linewidth}
 \[\routep{\prefix{w}{2}} = {\includegraphics[scale=0.8]{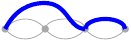}}\]
\end{subfigure}

\begin{subfigure}[ht]{\linewidth}
 \[\routep{\prefix{w}{3}} = {\includegraphics[scale=0.8]{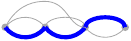}}\]
\end{subfigure}
\begin{subfigure}[ht]{\linewidth}
 \[\routep{\prefix{w}{4}} = {\includegraphics[scale=0.8]{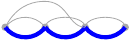}}\]
\end{subfigure}
    \caption{The maximal clique $\Delta_w=\{\routep{\prefix{w}{0}},\ldots,\routep{\prefix{w}{\sizes}}\}$ corresponding to the Stirling $(1,2,1)$-permutation $w=3221$.}
    \label{fig:clique_3221}
\end{figure}

\begin{lemma}\label{lem:max_clique}
The maximal simplices of $\triangDKK[\Gs]$ are exactly the simplices $\Delta_{w}$ where $w$ ranges over all Stirling $\s$-permutations.
\end{lemma}

\begin{proof}
Recall that by Theorem~\ref{thm:DKKsimplices} the maximal simplices of $\triangDKK[\Gs]$ are the simplices $\Delta_C$, where $C$ is a maximal clique of coherent routes of $(\Gs, \preceq)$.

Let $w$ be a Stirling $\s$-permutation. We will check that $\Delta_w$ is a clique of coherent routes of $(\Gs, \preceq)$. 
Let $1\leq i < i'\leq [\sizes]$ index two routes $\routep{\prefix{w}{i}}$ and $\routep{\prefix{w}{i'}}$ in $\Delta_w$.
Since $i<i'$, we have that $t_a(\prefix{w}{i})\leq t_a(\prefix{w}{i'})$ for all $a\in [n]$. 
Thus, for any vertex $v_{n+1-a}$ that appears in both routes $\routep{\prefix{w}{i}}$ and $\routep{\prefix{w}{i'}}$, we have that the incoming (respectively outgoing) edge of $\routep{\prefix{w}{i}}$ precedes the incoming (respectively outgoing) edge of $\routep{\prefix{w}{i'}}$ for the order $\preceq$.
Hence, the routes $\routep{\prefix{w}{i}}$ and $\routep{\prefix{w}{i'}}$ are coherent, and $\Delta_w$ is a clique for the coherence relation. 
Moreover, since $\Delta_w$ has $\sizes+1=\text{dim}(\fpol[\Gs])+1$ elements, it is a maximal clique, and corresponds to a maximal simplex in the DKK triangulation of $\fpol[\Gs]$.

Now, suppose that $w'$ is a Stirling $\s$-permutation distinct from $w$. We need to check that $\Delta_w\neq \Delta_{w'}$. 
Suppose that the minimal index $i\in [1,\sizes-1]$ such that $w_i \neq w'_i$ satisfies $w_i<w'_i$.  
Then $\routep{\prefix{w'}{i}}$ cannot belong to $\Delta_w$. 
Indeed, if we denote $a$ the value of $w_i$, we have that $e^{a}_{t_{a}(\prefix{w}{i})-1}$ is an edge of the route $\routep{\prefix{w'}{i}}$ but for any $j>i$, $t_{a}(\prefix{w}{j})\geq t_{a}(\prefix{w}{i})$, so the route $\routep{\prefix{w}{j}}$ does not contain this edge. 
Thus, the map $w \mapsto \Delta_w$ is an injection from Stirling $\s$-permutations to maximal simplices of $\triangDKK[\Gs]$.

Then, it follows from the bijection between $\s$-decreasing trees and maximal simplices of $\triangDKK[\Gs]$ (Remark~\ref{thm:bij_simplices_trees}) and the bijection between $\s$-decreasing trees and Stirling $\s$-permutations (Section~\ref{sec:combinatorics_of_the_s_permutahedron_stirling_permutations}) that this injection is a bijection. 
\end{proof}

\subsubsection{The (unoriented) Hasse diagram of the $\s$-weak order}
We show that the graph dual to the triangulation $\triangDKK[\Gs]$ coincides with the (unoriented) Hasse diagram of the $\s$-weak order.

\begin{theorem}\label{thm:cover_relations}
Let $s=(s_1, \ldots, s_n)$ be a composition.
Let $w$ and $w'$ be two Stirling $\s$-permutations. 
There is a cover relation between $w$ and $w'$ in the $\s$-weak order if and only if the simplices $\Delta_{w}$ and $\Delta_{{w'}}$ are adjacent in $\triangDKK[\Gs]$. 
\end{theorem}

\begin{figure}[t!]
    \centering
    \includegraphics[scale=0.7]{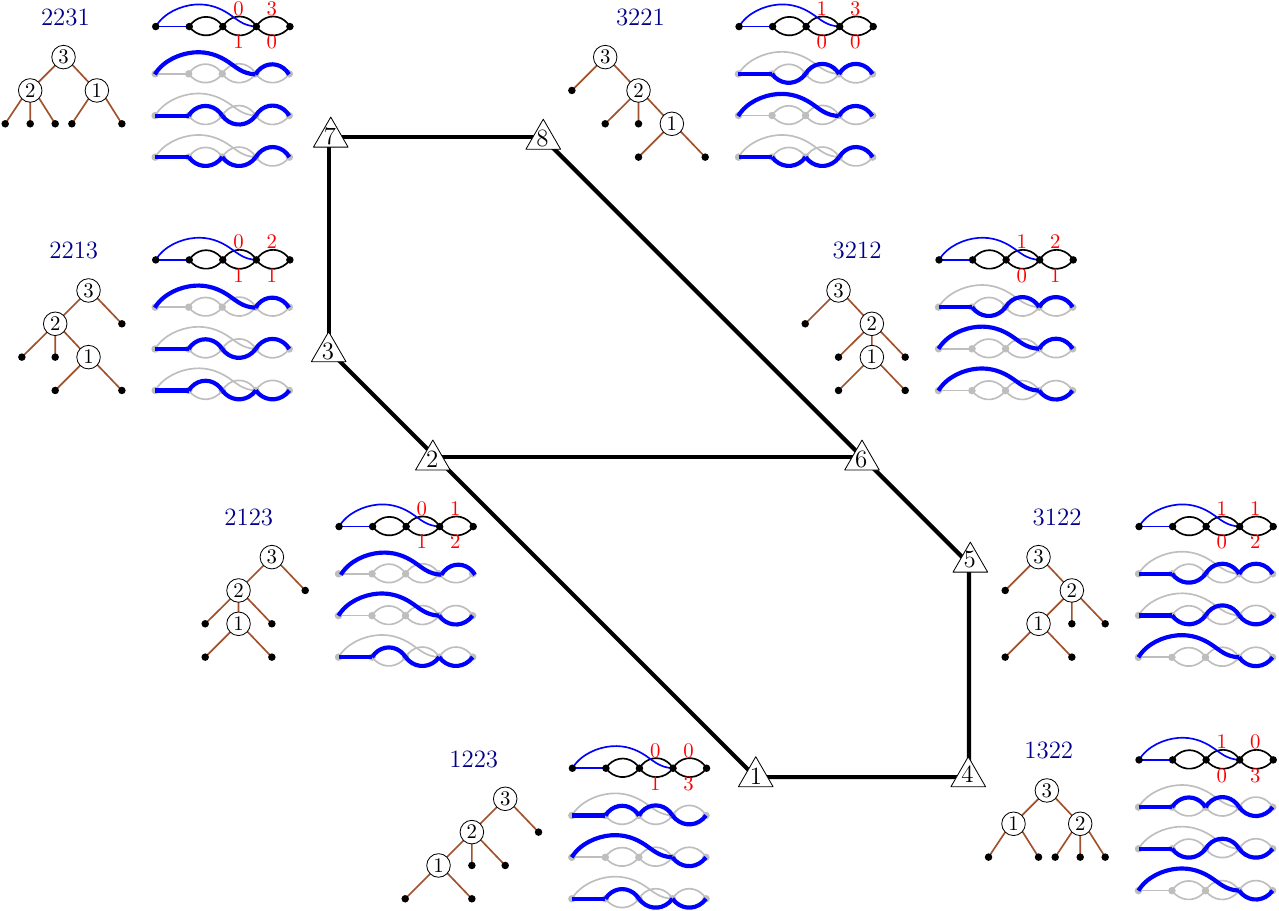}
    \caption{The $\s$-permutahedron for the case $\s=(1,2,1)$. The vertices are indexed by the following combinatorial objects: $\s$-decreasing trees, Stirling $\s$-permutations, maximal cliques of routes (omitting $\routep{\prefix{w}{0}}$ and $\routep{\prefix{w}{\sizes}}$), and integer $\mathbf{d}$-flows (in red on the topmost graph). They correspond to simplices or maximal mixed cells in our first and second realizations. }
 \label{fig:adjacency_graph_121}
\end{figure}

Figure~\ref{fig:adjacency_graph_121} shows the graph dual to the DKK triangulation of $\fpol[\Gs]$ for $s=(1,2,1)$, which corresponds to the (unoriented) Hasse diagram of the $(1,2,1)$-weak order. Note that in this Figure we are omitting the routes $\routep{\prefix{w}{0}}$ and $\routep{\prefix{w}{\sizes}}$ since both appear in $\Delta_{w}$ for every $w \in \mathcal{W}_{(1,2,1)}$.

\begin{proof}
Suppose that $w'$ is obtained from $w$ by a transposition along the ascent $(a,c)$.
It follows from Corollary~\ref{prop:cover_relations_multiperm} that $w=u_1 B_a c u_2$ and $w'=u_1 c B_a u_2$,  where $B_a$ is the $a$-block of $w$. 
We denote by $\ell(u)$ the length of a word $u$. 
For all $i\in [0, \ell(u_1)]$ and $i\in [\ell(u_1)+\ell(B_a)+1, \sizes]$, the routes $\routep{\prefix{w}{i}}$ and $\routep{\prefix{w'}{i}}$ are equal since $\prefix{w}{i}$ and $\prefix{w'}{i}$ have the same letters.
For all $i\in[\ell(u_1)+1, \ell(u_1)+\ell(B_a)-1]$, the routes $\routep{\prefix{w}{i}}$ and $\routep{\prefix{w'}{i+1}}$ are equal as well. 
Indeed, for such $i$ we have that $t_b(\prefix{w}{i})=t_{b}(\prefix{w'}{i+1})$ for all $b\in [n]\setminus \{c\}$ and since $0<t_a(\prefix{w}{i})<s_a$ (because we are reading the substring $B_a$) with $a<c$, the value of $t_c(\prefix{w}{i})$ (respectively $t_c(\prefix{w'}{i+1})$) does not play a role in the route $\routep{\prefix{w}{i}}$ (respectively $\routep{\prefix{w'}{i+1}}$). 
Hence, the vertices of $\Delta_w$ and $\Delta_{w'}$ differ only in one element: $\routep{\prefix{w}{\ell(u_1)+\ell(B_a)}}\in \Delta_w$ corresponding to the prefix $u_1B_a$ of $w$ and $\routep{\prefix{w'}{\ell(u_1)+1}}\in \Delta_{w'}$ corresponding to the prefix $u_1c$ of $w'$. 
This means that the corresponding simplices in the DKK triangulation of $\fpol[\Gs]$ share a common facet.

Reciprocally, suppose that the vertices of $\Delta_w$ and $\Delta_{w'}$ differ only in one element.
We denote $u_1$ the longest common prefix of $w$ and $w'$. We denote $a:=w_{\ell(u_1)+1}$, $c:=w'_{\ell(u_1)+1}$ and suppose that $a<c$. Let $B_a$ be the $a$-block of $w$. Then:
\begin{itemize}
    \item The substring $u_1$ contains no occurrence of $a$, because otherwise there would be a subword $aca$ in $w'$, which contradicts the $121$-pattern avoidance.
    \item The route $\routep{\prefix{w}{\ell(u_1)+\ell(B_a)}}$, corresponding to the prefix $u_1B_a$ of $w$, is in $\Delta_w$ but not in $\Delta_{w'}$.
    \item The route $\routep{\prefix{w'}{\ell(u_1)+1}}$, corresponding to the prefix $u_1c$ of $w'$, is in $\Delta_{w'}$ but not in $\Delta_w$.
\end{itemize}
Thus, the only possibility that $\Delta_w$ and $\Delta_{w'}$ differ only on these elements is that $w=u_1B_acu_2$ and $w'=u_1cB_au_2$, where $u_2$ is their longest common suffix. 
This means that there is an $\s$-tree rotation along the ascent $(a,c)$ between $w$ and $w'$.
\end{proof}

In this situation, we will say that the common facet of $\Delta_{w}$ and $\Delta_{{w'}}$ is associated to the transposition of $w$ along $(a,c)$. 
Note that such facets are exactly the interior facets (codimension-$1$ simplices) of $\triangDKK[\Gs]$.

\subsubsection{Higher faces of the $\s$-permutahedron}

We show that the faces of the $\s$-permutahedron other than vertices and edges are also encoded in the triangulation $\triangDKK[\Gs]$, after proving several technical results.

We say that a simplex of $\triangDKK[\Gs]$ is \defn{interior} if it is not contained in the boundary of the polytope $\fpol[\Gs]$. Otherwise it is in the boundary of $\triangDKK[\Gs]$.

\begin{lemma}\label{lem:routes_invsets}
    Let $w$ be a Stirling $\s$-permutation. Let $\route=\route(k,t,\delta)$ be a route of $\Gs$.
    Then, $\route$ is a vertex of $\Delta_w$ if and only if the multiset of inversions of $w$ satisfies the following inequalities:
    \begin{enumerate}
    \item $\#_w(k,i)\geq t$ for all $1\leq i < k$ such that $\delta_i=0$,
    \item $\#_w(k,i)\leq t$ for all $1\leq i < k$ such that $\delta_i=1$,
    \item $\#_w(j,i)=0$ for all $1\leq i < j <k$ such that $(\delta_i, \delta_j)=(1,0)$,
    \item $\#_w(j,i)=s_j$ for all $1\leq i < j <k$ such that $(\delta_i, \delta_j)=(0,1)$.
    \end{enumerate}
\end{lemma}

We say that the route $\route$ \defn{implies} these inequalities on inversion sets.

\begin{proof}
    ($\Rightarrow$) Suppose that $\route$ is a vertex of $\Delta_w$. It means that $\route=\routep{\prefix{w}{r}}$ for a certain $r\in[0, \sizes]$ and it conveys information on the prefix $u=\prefix{w}{r}$. More precisely, $t$ is the number of occurrences of $k$ in $u$, and for all $1\leq i < k$, the number of occurrences of $i$ in $u$ is either $0$ if $\delta_i=0$ or $s_i$ if $\delta_i=1$. This gives the announced inequalities on the inversion set of $w$.
    
    ($\Leftarrow$) Reciprocally, suppose that the inversion set of $w$ satisfies these inequalities. Then there is a prefix $u=\prefix{w}{r}$ of $w$ that contains no occurrence of $i$ for $i$ such that $\delta_i=0$, all $s_i$ occurrences of $i$ for $i$ such that $\delta_i=1$, and exactly $t$ occurrences of $k$. Then this prefix is exactly the one associated to the route $\routep{u}=\route(k,t,\delta)=\route$.
\end{proof}

    Let $w$ be a Stirling $\s$-permutation, $A$ a subset of ascents of $w$, and $1\leq a < c\leq n$ such that $\#_w(c,a)<s_c$. We say that the pair $(a,c)$ is \defn{$A$-dependent} in $w$ if there is a sequence $a\leq b_1< \ldots <  b_k< b_{k+1}=c$ such that: 
    \begin{itemize}
        \item[(i)] $b_1$ is the greatest letter strictly smaller than $c$ such that the $a$-block is contained in the $b_1$-block,
        \item[(ii)] for all $i\in[k-1]$, the $b_{i}$-block is directly followed by the $b_{i+1}$-block,
        \item[(iii)] the $b_k$-block is directly followed by an occurrence of $c$,
        \item[(iv)] $(b_i,b_{i+1})\in A$ for all $i\in [k]$.
    \end{itemize}

Note that in particular, every ascent $(a,c)$ in $A$ is $A$-dependent taking $k=1$, $b_1=a$ and $b_2=c$.

For example for the Stirling $\s$-permutation $w=33725455716$ and $A=\{(2,5), (5,7), (1,6)\}$ there is an $A$-dependency between $2$ and $7$ given by the sequence $b_1=2$, $b_2=5$ and $b_3=7$ but there is no $A$-dependency between $2$ and $6$ since the second occurrence of $7$ does not form a block and it is followed by $1$.

\begin{proposition}\label{prop:transitiveclosure}
    Let $w$ be a Stirling $\s$-permutation and $A$ a subset of its ascents. Then we have that
    $$\inv(w+A)=\begin{cases}
        \#_w(c,a)+1 & \text{if $(a,c)$ is $A$-dependent in }w\\
        \#_w(c,a)&\text{otherwise.}
    \end{cases}$$
\end{proposition}

\begin{example}
    If $w=33725455716$ and $A=\{(2,5), (5,7), (1,6)\}$, then the resulting Stirling $s$-permutation is $w+A=33775245561$ and the pairs whose inversion number has been increased by $1$ are $\{(2,5), (2,7), (4,7), (5,7), (1,6)\}$.
\end{example}

\begin{proof}
Let be $I$ the multiset of inversions defined by
    $$\#_I(c,a):=\begin{cases}
        \#_w(c,a)+1 & \text{if $(a,c)$ is $A$-dependent in }w\\
        \#_w(c,a)&\text{otherwise.}
    \end{cases}$$

Note that if  $(a,c)$ is $A$-dependent and $d>c$ this implies that $\#_I(d,c)=\#_I(d,a)$. Indeed, in this case either both $(a,d)$ and $(c,d)$ are $A$-dependent or both are not.

    We will prove that $I=\inv(w+A)$ by showing that $I$ is the smallest transitive multiset of inversions that contains $\inv(w)+A$. Following \cite[Definition 2.3]{CP20} or \cite[Definition 1.14]{CP22}, by \defn{transitivity} of $I$ we understand that if $a<b<c$, then $\#_I(b,a)=0$ or $\#_I(c,a)\geq \#_I(c,b)$.
    
    First, it is clear that $I$ contains $\inv(w)+A$ since every pair in $A$ is $A$-dependent.
    
    Let us show that any transitive multiset of inversions $I'$ that contains $\inv(w)+A$ necessarily contains $I$. Note that for such $I'$ we have $\#_{I'}(c,a)\geq \#_w(c,a)$ for all pairs $(a,c)$. 
    Let $(a,c)$ be $A$-dependent with an associated sequence $a\leq b_1< \ldots <  b_k< b_{k+1}=c$ and we proceed by induction on $k$. 
    
    If $k=1$, we have that:
    \begin{itemize}
    \item either $b_1=a$ and $(a,c)$ is in $A$, and directly $\#_{I'}(c,a)\geq \#_w(c,a)+1$, 
    \item or $a<b_1<c$, which in this case $\#_{I'}(b_1,a)\geq \#_w(b_1,a)>0$ since the $a$-block is contained in the $b_1$-block in $w$. We get that \begin{equation}\label{eq:transitivity+A+increasing_pair}
        \#_{I'}(c,a)\geq \#_{I'}(c,b_1)\geq \#_w(c,b_1)+1=\#_w(c,a)+1
    \end{equation} where the first inequality comes from transitivity, the second from the previous case since $(b_1,c)\in A$ and the last equality again because the $(a,b_1)$ are $A$-dependent.
    \end{itemize}
    Suppose that $k>1$. Then the induction hypothesis implies that $\#_{I'}(b_k,a)\geq \#_w(b_k,a)+1 > 0$. Just like in equation~\eqref{eq:transitivity+A+increasing_pair}, applying transitivity to $a<b_k<c$ and using that $(b_k,c)\in A$ and $(a,b_k)$ is $A$-dependent (so there cannot be any occurrance of $c$ between $a$ and $b_k$) gives us the inequalities 
    $\#_{I'}(c,a)\geq \#_{I'}(c,b_k)\geq \#_{w}(c,b_k)+1=\#_w(c,a)+1.$
    
    Finally, we check that $I$ is indeed transitive. Let $1\leq a < b < c \leq n$ and $\#_I(b,a)>0$. We need to check that $\#_I(c,a)\ge \#_I(c,b)$.
    
    \textbf{\quad Case 1:} if $\#_w(b,a)=0$, then $(a,b)$ is $A$-dependent and $\#_I(c,a)=\#_I(c,b)$. 
    
    \textbf{\quad Case 2:} Suppose that $\#_w(b,a)>0$.

    \textbf{\quad\quad Case 2.1:} If $\#_I(c,b)=\#_w(c,b)$, due to the inclusion $\inv(w)\subset I$ and the transitivity of $\inv(w)$ we have that $\#_I(c,a)\geq \#_w(c,a)\geq \#_w(c,b)=\#_I(c,b)$.

    \textbf{\quad\quad Case 2.2:}
    Suppose that $\#_I(c,b)=\#_w(c,b)+1$ i.e.\ $(b,c)$ is $A$-dependent. If $\#_w(c,a)\geq \#_w(c,b)+1$, we have $\#_I(c,a)\geq \#_w(c,a)\geq \#_w(c,b)+1 =\#_I(c,b)$. Otherwise, we have $\#_w(c,a)=\#_w(c,b)=:i$. It follows from the assumption $\#_w(b,a)>0$ that the $a$-block appears in $w$ between the first occurrence of $b$ and the $i$-th occurrence of $c$. This implies that $(a,c)$ is also $A$-dependent, with a corresponding sequence that is included in the one giving the the $A$-dependency of $(b,c)$. The two $A$-dependencies together with the transitivity of $w$ for $a<b<c$ imply that $\#_I(c,a)=\#_w(c,a)+1\geq\#_w(c,b)+1=\#_I(c,b)$.\end{proof}

    Let $(w, A)$ be a face of $\spermcombi$. 
    We define \defn{$\Delta_{(w, A)}$} to be the following intersection of facets of $\Delta_w$:
\begin{equation} \label{def:Delta_(T,A)}
\Delta_{(w,A)}:=\bigcap_{(a,c)\in A} \left\{\Delta_w\cap \Delta_{w'} \, | \, w' \text{ is the transposition of $w$ along $(a,c)$} \right\},
\end{equation}
    and $\Delta_{(w,A)}:=\Delta_w$ if $A=\emptyset$.
    
Note that the $|A|$ routes that are in $\Delta_w$ and not in $\Delta_{(w,A)}$ correspond to the prefixes of $w$ that end at an ascent in $A$. 

\begin{lemma}\label{lem:Delta_(T,A)}
    Let $(w, A)$ be a face of $\spermcombi$ and $w'$ a Stirling $\s$-permutation. We denote by $[w, w+A]$ the interval of the $\s$-weak order defined by $w$ and $w+A$.
    
    Then, $\Delta_{(w,A)}\subseteq \Delta_{w'}$ if and only if $w'\in [w, w+A]$.
\end{lemma}

\begin{proof}
Recall that the inversion set of $w+A$ is described in Proposition~\ref{prop:transitiveclosure} and that $w'\in [w, w+A]$ if and only if its inversion set satisfies that for all $1\leq a < c \leq n$, $\#_w(c,a)\leq \#_{w'}(c,a)\leq \#_{w+A}(c,a)$. 
We show that these inequalities are exactly the ones implied by the union of routes that give vertices of $\Delta_{(w,A)}$, in the sense of Lemma~\ref{lem:routes_invsets}.

Let $(a,c)$ be a pair with $\#_w(c,a)=t$. We have to show these three inequalities:
\begin{enumerate}
    \item There is a route $R$ in $\Delta_{(w,A)}$ such that $R\in \Delta_{w'}$ implies the inequality $\#_{w'}(c,a)\geq t$. 
    We can take the route that corresponds to the first prefix of $w$ containing the $t$-th occurrence of $c$ that does not end at an ascent in $A$. Such a prefix cannot contain any $a$ since $a<c$.
    \item There is a route $R$ in $\Delta_{(w,A)}$ such that $\#_{w'}(c,a)\leq t$ for all $w'$ with $R\in\Delta_{w'}$ if and only if $\#_{w+A}(c,a)=t$, that is, the pair $(a,c)$ is not $A$-dependent. 
    Indeed, this inequality is only implied by routes that contain the edges $e^c_t$ and $e^a_{s_a}$.
    Such routes in $\Delta_w$ correspond to prefixes in $w$ that contain the $a$-block and the $t$-th occurrence of $c$ and that do not end inside a $b$-block for any $b<c$. The pair $(a,c)$ is $A$-dependent exactly when all such prefixes end at a descent in $A$, so the corresponding routes are removed in $\Delta_{(w,A)}$.
    \item If $t+1< s_c$ and $\#_{w+A}(c,a)=t+1$, i.e.\ $(a,c)$ is an $A$-dependent pair, there is a route $R$ in $\Delta_{(w,A)}$ such that $R\in \Delta_{w'}$ implies $\#_{w'}(c,a)\leq t+1$.
    Indeed, we can take the route that corresponds to the prefix of $w$ that ends at the $(t+1)$-th occurrence of $c$. Since $t+1<s_c$,$~c$ appears afterwards so this prefix does not end at an ascent. (Note that if $t+1=s_c$ there is no need to check that $\#_{w'}(c,a)\leq s_c$). \qedhere
\end{enumerate}
\end{proof}

Lemma \ref{lem:Delta_(T,A)} leads to the following alternative characterization of $\Delta_{(w,A)}$.
\begin{corollary}
    $\Delta_{(w,A)}=\bigcap_{w' \in [w, w+A]} \Delta_{w'}$.
\end{corollary}

\begin{lemma}\label{cor:interiorsimplicesDKK}
If $C$ is a clique of routes of $(\Gs, \preceq)$ that contains $\route(n+1,1,(0)^n), \route(n+1, 1, (1)^n)$ and at least one route that starts with $e$ for each source-edge $e$ that is not $(v_{-1}, v_0)$, then $\Delta_C$ is in the interior of $\triangDKK[\Gs]$.
\end{lemma}

\begin{proof}
    Suppose that $\Delta_C$ is a boundary simplex of $\triangDKK[\Gs]$. Then it is contained in a facet that is in the boundary of $\triangDKK[\Gs]$. This facet corresponds to a clique of the form $\Delta_w\setminus R$, where $w$ is a Stirling $\s$-permutation and $R$ is a route of $\Delta_w$ that does not correspond to an ascent nor a descent of $w$. Hence, either $R\in\{\route(n+1,1,(0)^n), \route(n+1,1,(1)^n)\}$, or $R$ corresponds to a prefix $\prefix{w}{i}$ such that $w_i=w_{i+1}$. In this case, suppose that $w_i$ is the $t$-th occurrence of $c$ in $w$. Then $R$ is the only route of $\Delta_w$ that starts with the edge $e^c_t$. In any case, since $C\subseteq \Delta_w\setminus R$, it does not satisfy the condition of the lemma.
\end{proof}

\begin{corollary}\label{cor:Delta(w,A)_interior}
    Let $w$ be a Stirling $\s$-permutation and $A$ a subset of its ascents. Then $\Delta_{(w,A)}$ is an interior simplex of $\triangDKK[\Gs]$.
\end{corollary}

\begin{proof}
It is sufficient to show that $\Delta_{(w,A)}$ contains $\route(n+1,1,(0)^n), \route(n+1, 1, (1)^n)$ and at least one route that starts with $e$ for each source-edge $e$ that is not $(v_{-1}, v_0)$.

First, it is clear that $\route(n+1, 1, (0)^n)$ and $\route(n+1, 1, (1)^n)$ are in $\Delta_{(w,A)}$ since they do not correspond to ascents in $w$.

    Let $c\in[n]$ and $t\in[s_{c}-1]$. Then the prefix of $w$ that ends with the $t$-th occurrence of $c$ corresponds to a route $R$ that contains the edge source-edge $e^c_t$. Moreover, there cannot be an ascent of $w$ after this prefix since there are still occurrences of $c$ afterwards. Thus the route $R$ is not removed from $\Delta_w$ to $\Delta_{(w,A)}$.
\end{proof}

\begin{theorem}\label{thm:bij_interiorfacesDKK_facessperm}
    The map $(w,A) \mapsto \Delta_{(w,A)}$ induces a poset isomorphism between the face poset of the $\s$-permutahedron $\spermcombi$ and the set of interior simplices of $\triangDKK[\Gs]$ ordered by reverse inclusion.
\end{theorem}

\begin{proof}
    The fact that all $\Delta_{(w,A)}$ are interior simplices of $\triangDKK[\Gs]$ is stated in Corollary~\ref{cor:Delta(w,A)_interior}. 
    The injectivity follows from Lemma~\ref{lem:Delta_(T,A)}. 
    
    Let us show the surjectivity. 
    Let $F$ be an interior simplex of $\triangDKK[\Gs]$. Let $w$ be a Stirling $\s$-permutation that is minimal for the $\s$-weak order with respect to the condition that $F\subseteq \Delta_w$. Then, $F$ is an intersection of facets of $\Delta_w$. These facets correspond to certain transpositions involving $w$. We denote by $A$ the set of ascents corresponding to these transpositions. The minimality of $w$ implies that all elements in $A$ are ascents (and not descents) of $w$. Thus, $F=\Delta_{(w,A)}$, and the choice of $w$ was unique. 

    Finally, let $w, w'$ be Stirling $\s$-permutations and $A,A'$ subsets of their respective ascents. Lemma~\ref{lem:Delta_(T,A)} implies that $[w, w+A]\subseteq [w', w'+A']$ if and only if $\Delta_{(w', A')} \subseteq \Delta_{(w,A)}$, which proves that the map is a poset isomorphism.
\end{proof}

Just as how the minimal elements of the face poset of $\spermcombi$ have a characterization as the maximal cliques of $\triangDKK[\Gs]$, the maximal elements of the face poset also have an explicit characterization in terms of cliques.

\begin{corollary}\label{cor:maximal_faces_perm}
A simplex $\Delta_C$ of $\triangDKK[\Gs]$ corresponds with a maximal interior face of $\spermcombi$ if and only if $C$ is a clique of size $|s|-n+2$  that satisfies the following:
\begin{itemize}
    \item $\routep{\prefix{w}{0}}$ and $\routep{\prefix{w}{|s|}}$ are in $C$, and
    \item each source-edge of $\Gs$ that is different from $(v_{-1},v_0)$ is contained in exactly one route in $C$.
\end{itemize}
\end{corollary}
\begin{proof}
We first note that for each $i\in [n]$, the graph $\Gs$ has $s_i-1$ source-edges, so $\Gs$ indeed has $\sum_{i=1}^n (s_i-1) = |s|-n$ source-edges that are not $(v_{-1},v_0)$.
By Lemma~\ref{cor:interiorsimplicesDKK}, a clique $C$ with the above stated properties corresponds with a maximal interior face of $\spermcombi$.

Conversely, let $(w, A)$ be a maximal face of $\spermcombi$.
We will check that $C = \Delta_{(w,A)}$ satisfies the specified properties.
Let $N \subset [0,|s|]$ denote the set of non-ascent positions in $w$, so that $C = \cup_{j\in N} \routep{\prefix{w}{j}}$.  

Observe that since $w$ is $121$-avoiding, if it has $n-1$ ascents, then the ascents are of the form $(i,c_i)$ where $i < c_i$ for each $i\in [n-1]$.
Moreover, for each $i\in [n-1]$, it is the $s_i$-th occurrence of $i$ in $w$ which produces an ascent pair in $w$.
Therefore, the set $N\backslash \{0, |s|\}$ indexes the first $s_i-1$ occurrences of $i$ in $w$.

Now suppose $j\in N \backslash\{0, |s|\}$ is a non-ascent position of $w$  so that $\routep{\prefix{w}{j}} \in C$. We denote $a$ the letter $w_j$.
If $w_j$ is the $k$-th occurrence of $a$ in $w$ for some $k\in[s_{a}-1]$, then the route $\routep{\prefix{w}{j}}$ contains the proper source-edge $e^{a}_k$.
Lastly, since $|N\backslash \{0, |s| \}| = |s|-n$, then $C$ has the desired properties.
\end{proof}

\subsubsection{On the \texorpdfstring{$h$}{h} and $h^*$-polynomials}
\label{sec:hstar}

Given a (simplicial, polytopal) complex $\Delta$ one can define its $f$- and $h$-polynomials as follows. The \defn{$f$-polynomial} of $\Delta$ is defined as
$$f_{\Delta}(x)=\sum_{F\in \Delta}x^{\dim(F)},$$
where the sum is over all the faces $F$ of $\Delta$. Then  its \defn{$h$-polynomial} is defined by the relation
$$f(x)=h(x+1).$$

Let $\setperms(k)$ denote the subset of Stirling $\s$-permutations with $k$ descents.
We generalize the notion of the Eulerian polynomial and define the \defn{$\s$-order Eulerian polynomial} by $A_{\s}(x)= \sum_{k\geq0} |\setperms(k)|x^k$.
This recovers the classical Eulerian polynomial when $\s=(1,\ldots, 1)$ and the second-order Eulerian polynomial (see for example \cite[Section 6.2]{GKP94}) when $\s=(2,\ldots, 2)$
\footnote{Savage and Visontai~\cite{SavVison} have previously defined a notion of $\s$-Eulerian polynomials in the context of {\em $\s$-lecture hall polytopes} that is different from $A_{\s}(x)$.}.

The $f$-polynomial of the $\s$-permutahedron is
\begin{align*}
    f_{\spermcombi}(x)&=\sum_{\substack{w\in \setperms\\A\subseteq \ASC(w)}}x^{|A|}\\
    &=\sum_{w\in \setperms}(1+x)^{|\ASC(w)|}\\
    &=\sum_{k\geq 0} |\setperms(k)|(x+1)^k\\
    &=A_{\s}(x+1),
\end{align*}
where $\ASC(w)$ denotes the set of ascents of a Stirling $\s$-permutation $w$.

In the second-to-last equality we used the fact that the number of elements in $\setperms$ with $k$ ascents is equal to the number of elements in $\setperms$ with $k$ descents, which can be seen by reading $w$ in reverse.

\begin{proposition}
    The $h$-polynomial of $\spermcombi$ is $A_{\s}(x)$.
\end{proposition}

\begin{example}
As an example, the reader can compute the $h$-polynomial of $\spermcombi$ when $\s=(1,2,1)$ from Figure \ref{fig:adjacency_graph_121_classic} to get $h_{\spermcombi}(x)=1+5x+2x^2$.
\end{example}

Ehrhart showed in~\cite{E62} that the lattice point enumerator $|tP\cap \mathbb{Z}^d|$ of a $t$ dilated $d$-dimensional lattice polytope $P$ is a polynomial in $t$, known as the \defn{Ehrhart polynomial} of $P$. In addition, he showed that the generating series has the form
$$1+\sum_{t\ge 1}|tP\cap \mathbb{Z}^d|x^t=\dfrac{h^*(x)}{(1+x)^{d+1}},$$
where $h^*(x)$ is a polynomial of degree $d$ known as the \defn{$h^*$-polynomial} of $P$. This polynomial was shown by Stanley to have nonnegative coefficients~\cite{s80} and it turns out that the $h^*$-polynomial coincides with the $h$-polynomial of a shellable unimodular triangulation (see for example~\cite{BS07}). A triangulation is said to be \defn{shellable} if its facets can be totally ordered $F_1,\ldots, F_m$ such that for any pair $p<q$, there exists $r < q$ satisfying $|F_r\cap F_q|=|F_q|-1$ and $F_p\cap F_q\subseteq F_r \cap F_q$. 

\begin{theorem}
   The $h^*$-polynomial of $\fpol[\Gs]$ is $A_{\s}(x)$. 
\end{theorem}
\begin{proof}
Similar to the proof of~\cite[Lemma 6.1]{vBGDLMCY21}, we will show that any linear extension of the $\s$-weak order gives a shelling $\Delta_{w^{(1)}},\ldots, \Delta_{w^{(|\setperms|)}}$ of $\triangDKK[\Gs]$.

Fix a linear extension of the $\s$-weak order and let $p<q$ such that $w^{(p)} < w^{(q)}$ in the linear extension. 
Let $r<q$ be such that  $w^{(p)}\wedge w^{(q)} <w^{(r)} \lessdot w^{(q)}$ in the $\s$-weak order. 
It is clear that $|\Delta_{w^{(r)}} \cap \Delta_{w^{(q)}}| = |\Delta_{w^{(q)}}|-1$ since $\Delta_{w^{(r)}}$ and $\Delta_{w^{(q)}}$ intersect in a facet.
Now suppose that there is a route $R=R(k,t,\delta)$ such that $R \in \Delta_{w^{(p)}}\cap \Delta_{w^{(q)}}$ but $R \not \in \Delta_{w^{(r)}}$. By Lemma \ref{lem:routes_invsets} one of the following cases will occur:
 \begin{enumerate}
    \item for some $1\leq i < k$ such that $\delta_i=0$ we will have $\#_{w^{(r)}}(k,i)< t$ but $\#_{w^{(p)}}(k,i)\geq t$ and $\#_{w^{(q)}}(k,i)\geq t$, or
    \item for some $1\leq i < k$ such that $\delta_i=1$ we will have $\#_{w^{(r)}}(k,i)> t$, but $\#_{w^{(p)}}(k,i)\leq t$ and $\#_{w^{(q)}}(k,i)\leq t$, or
    \item for some $1\leq i < j <k$ such that $(\delta_i, \delta_j)=(1,0)$ we have that
    $\#_{w^{(r)}}(j,i)=s_j$ but $\#_{w^{(p)}}(j,i)=\#_{w^{(q)}}(j,i)=0$, or
    \item for some $1\leq i < j <k$ such that $(\delta_i, \delta_j)=(0,1)$ we have that
    $\#_{w^{(r)}}(j,i)=0$ but $\#_{w^{(p)}}(j,i)=\#_{w^{(q)}}(j,i)=s_j$.
    \end{enumerate}
Any one of these cases will contradict the fact that
    $$\inv(w^{(p)})\cap \inv(w^{(q)})=\inv(w^{(p)}\wedge w^{(q)})\subseteq \inv(w^{(r)}) \subseteq \inv(w^{(q)}).$$
Hence, $R \in \Delta_{w^{(r)}}$ and we have $\Delta_{w^{(p)}}\cap \Delta_{w^{(q)}}\subseteq \Delta_{w^{(r)}}\cap \Delta_{w^{(q)}}$.

Now, let 
\[V_q=\{ v\in \Delta_{w^{(q)}} \mid v \hbox{ is a vertex in } \Delta_{w^{(q)}} \hbox{ and } \Delta_{w^{(q)}}\backslash v \subseteq \Delta_{w^{(p)}} \hbox{ for some } 1 \leq p< q\}.
\]
Then the $k$-th coefficient of the $h$-polynomial is $|\{q \mid |V_q|=k, 1\leq q \leq |\setperms| \} |$, which is the number of Stirling $\s$-permutations that cover exactly $k$ elements in the $\s$-weak order and hence is the number of Stirling $\s$-permutations with exactly $k$ descents.
\end{proof}

\begin{example}
Again as an example, the reader can confirm that the $h^*$-polynomial of $\fpol[\Gs]$ when $\s=(1,2,1)$ is $h^*_{\fpol[\Gs]}(x)=1+5x+2x^2$.
\end{example}

\section{Cayley trick and mixed subdivisions}
\label{sec:Cayley_trick}

The Cayley trick allows us to give another geometric realization of the $\s$-permutahedron as a fine mixed subdivision of a $n$-dimensional polytope (or even as a $(n-1)$-dimensional one).
This technique was first developed by Sturmfels in~\cite[Section 5]{Sturmfels94} for coherent subdivisions and by Humber, Rambau, and Santos in~\cite{HRS00} for arbitrary subdivisions.
It was applied to flow polytopes by M\'esz\'aros and Morales in~\cite[Section 7]{MM19}. 
We slightly modify their work for our special case of the flow polytope $\fpol[\Gs]$. To this end we need some basic definitions.

\subsection{Background on the Cayley trick}

Consider the polytopes $P_1,\ldots,P_k$ in $\RR^d$. Their \defn{Minkowski sum} is the polytope \[P_1+\ldots+P_k:=\{x_1+\cdots + x_k\,|\,x_i\in P_i\}.\] 
For the Minkowski sum of $k$ copies of a polytope $P$ we simply write $kP$. The cells of this subdivision are called \defn{Minkowski cells} and are obtained via sums $\sum B_i$ where $B_i$ is the convex hull of a subset of vertices of $P_i$. A \defn{mixed subdivision} of a Minkowski sum is a collection of Minkowski cells such that their union covers the Minkowski sum and they intersect properly as Minkowski sums (see \cite[Definition~1.1]{Santos05}). A \defn{fine mixed subdivision} is a minimal mixed subdivision via containment of its summands.

Let $e_1,\ldots,e_k$ be a basis of $\RR^k$. We call the polytope 
\[\cC(P_1,\ldots,P_k):=\emph{conv}(\{\mathbf{e_1}\}\times P_1, \ldots, \{\mathbf{e_k}\}\times P_k)\subseteq \RR^{k}\times \RR^d\] 
the \defn{Cayley embedding} of $P_1,\ldots,P_k$.

\begin{proposition}[{The Cayley trick \cite[Section 5]{Sturmfels94}}]\label{prop:cayleytrick}
Let $P_1,\ldots,P_k$ be polytopes in $\RR^d$. The regular polytopal subdivisions (respectively triangulations) of $\cC(P_1,\ldots,P_k)$ are in bijection with the regular mixed subdivisions (respectively fine mixed subdivisions) of $P_1+\ldots+P_k$. 
\end{proposition}

A concrete way to have this bijection (see the ``one-picture-proof'' \cite[Figure~1]{HRS00}) is to intersect a subdivision of $\cC(P_1, \ldots, P_k)$ with the subspace $(\frac{1}{k}, \ldots, \frac{1}{k})\times \RR^n$ of $\RR^k\times \RR^n$. Up to dilation by the factor $k$ we obtain a mixed subdivision of $P_1+\ldots +P_k$. 

For regular subdivisions, this also gives a way to obtain an admissible height function for a mixed subdivision of $P_1+\ldots+P_k$ from an admissible height function for a subdivision of $\cC(P_1, \ldots, P_k)$, see Section~\ref{subsec:height_functions}.

\begin{remark}
  Note that the Cayley trick induces a poset isomorphism between the interior faces of the subdivision of $\cC(P_1, \ldots, P_k)$ and the interior faces of the corresponding mixed subdivision of $P_1+\ldots+P_k$ (both sets of faces being ordered by inclusion).  
\end{remark}

\subsection{The sum of cubes realization}\label{subsec:realization_mixedsubdiv}
To apply the Cayley trick to our triangulation $\triangDKK[\Gs]$ of the flow polytope $\fpol[\Gs]$, we need to describe it as the Cayley embedding of some lower-dimensional polytopes. 
Recall that $\fpol[\Gs]$ lives in the space of edges of the graph $\Gs$. We parameterize this space as $\RR^{p} \times \RR^{2n}$, where $p=1 +\sum_{i=1}^n (s_i-1)$ and $\RR^p$ corresponds to the space of source-edges and $\RR^{2n}$ to the space of bumps and dips (edges of $\oruga$, see Definition~\ref{def:Gs}).
Moreover, for all $i\in [n]$ and for any point in $\fpol[\Gs]$, (i.e.\ a flow of $\fpol[\Gs]$), we have that the sum of its coordinates along edges $e^i_0$ and $e^i_{s_{i}}$ is determined by the coordinates along the source-edges $e^k_t$ for $k\in[i+1, n+1]$, $t\in [s_k-1]$.
Thus, $\fpol[\Gs]$ is affinely equivalent to its projection on the space $\RR^p\times \RR^{n}$ where $\RR^{n}$ corresponds to the space of edges $e^i_0$ for $i\in [n]$.

With this parameterization, the indicator vector of the route of $\Gs$ denoted $\route(k,t,\delta)$ (as in the discussion after Definition~\ref{def:Gs}) with $k\in[n+1]$, $t\in[s_k-1]$ and $\delta\in \{0,1\}^{k-1}$  is:
$$e^{k}_{t}\times \sum_{i\in[k-1], \, \delta_i=0} e^i_0.$$ 
Thus, if we denote by $\square_{k-1}$ these $(k-1)$-dimensional hypercubes with vertices $\{0,1\}^{k-1}\times 0^{n-k+1}$ embedded in $\RR^{n}$, we see that $\fpol[\Gs]$ is the Cayley embedding of $\square_n$ and $\square_{k-1}$ repeated $s_{k}-1$ times for $k\in [n]$.

We denote by \defn{$\subdivCay$} the fine mixed subdivision of the Minkowski sum of hypercubes $\square_{n}+\sum_{i=1}^{n} (s_i-1)\square_{i-1} \subseteq \RR^n$ obtained by intersecting the triangulation $\triangDKK[\Gs]$ (projected onto $\RR^p\times \RR^n$) with the subspace ${\left\{\frac{1}{p}\right\}}^p\times \RR^{n}$.

The following theorem follows directly from the Cayley trick (Proposition~\ref{prop:cayleytrick}), and the isomorphism between the face poset of $\spermcombi$ and the interior simplices of the DKK triangulation given in Theorem~\ref{thm:bij_interiorfacesDKK_facessperm}.

\begin{theorem}\label{thm:bij_mixed_subdiv}
The face poset of the $\s$-permutahedron $\spermcombi$ is isomorphic to the set of interior cells of $\subdivCay$ ordered by reverse inclusion.
\end{theorem}

In particular, the $\s$-decreasing trees are in bijection with the maximal cells of $\subdivCay$.

\begin{remark}
    We can use a different parameterization of the space where $\fpol[\Gs]$ lives by considering the cube $\square_n$ as the Cayley embedding of two hypercubes $\square_{n-1}$, or equivalently intersect $\RR^n$ with the hyperplane $x_n=\frac{1}{2}$. This allows us to lower the dimension and obtain a fine mixed subdivision of the Minkowski sum of hypercubes $(s_n +1)\square_{n-1}+\sum_{i=1}^{n-1} (s_i-1)\square_{i-1}$. We use this representation for the figures.
\end{remark}

Figure~\ref{fig:mixed_cell_3221} shows the mixed cell corresponding to the Stirling $(1,2,1)$-permutation $w=3221$, obtained from the clique $\Delta_w$ with the Cayley trick.
Figure~\ref{fig:mixed_subdivision} shows the entire mixed subdivision for the case $s=(1,2,1)$.
Both figures are represented in the coordinate system $(e^2_0, e^3_0)$. 
In Figure~\ref{fig:adjacency_graph_121} the dual graph of the cells of this mixed subdivision is portrayed with edges oriented perpendicular to each inner wall.

\begin{figure}
\centering
  \begin{subfigure}[b]{0.5\textwidth}
\centering
\includegraphics[scale=0.85]{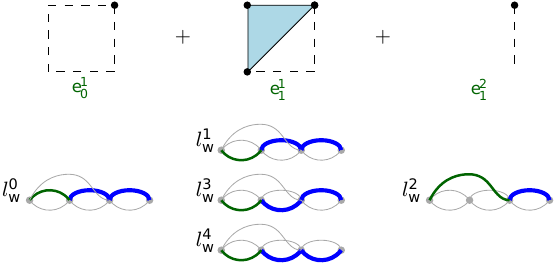}
    \caption{}
        \label{fig:mixed_cell_3221}
\end{subfigure}
\qquad\qquad 
\begin{subfigure}[b]{0.3\textwidth}
  \centering
\includegraphics[scale=0.7]{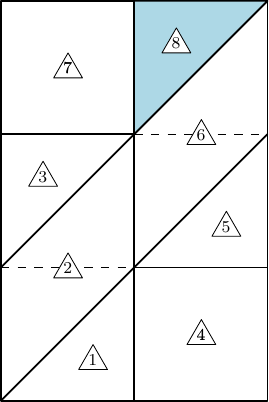}
  \caption{}
    \label{fig:mixed_subdivision}
\end{subfigure}
     \caption{(a) Summands of the Minkowski cell corresponding to $w=3221$ together with their corresponding routes in the clique $\Delta_w$. (b) Mixed subdivision of $2\square_2+\square_1$ corresponding dually to the $(1,2,1)$-permutahedron. The cells are numbered according to Figure~\ref{fig:adjacency_graph_121}. The highlighted cell in blue corresponds to $w=3221$ as obtained in Figure~\ref{fig:mixed_cell_3221}.}
    \label{fig:comparing flows PS and gen PS}
\end{figure}

\section{Intersection of tropical hypersurfaces}\label{sec:realization_polyhedral_complex}

In the two realizations that we provided in Sections~\ref{subsec:realization_flowpolytopes} and~\ref{subsec:realization_mixedsubdiv}, the $\s$-decreasing trees index the maximal cells of a polytopal complex. However, Conjecture~\ref{conj:s-permutahedron} asks for a polytopal complex where the $\s$-decreasing trees index the vertices.

In this section, we explain how to dualize our previous realizations in order to obtain such a polytopal realization and fully answer the conjecture for strict compositions. 
Tropical geometry offers a nice setting to dualize regular polyhedral subdivisions that moreover behaves well with the Cayley trick.

\subsection{Background on tropical dualization}\label{sec:background_tropicalization} This section is based on the work of Joswig in~\cite{Joswig16} and \cite[Chapter 1]{Joswig21}.

Let $\cA=\{\mathbf{a}^1, \ldots, \mathbf{a}^m\}$ be a point configuration in $\RR^d$ with integer coordinates, and $\cS$ a subdivision of $\cA$. 

The subdivision $\cS$ is said to be \defn{regular} if there is a function $\height: [m]\to \RR, i\mapsto \height^i$ such that the faces of $\cS$ are the images of the lower faces of the lift of $\cA$ (the polytope with vertices $(\mathbf{a}^i, \height^i) \in \RR^{d+1}$ for $i \in [m]$) by the projection that omits the last coordinate. 
In this case, the function $\height$ is called an \defn{admissible height function} for $\cS$.

Such a point configuration together with a height function $\height$ is associated to the \defn{tropical polynomial} (in the min-plus algebra): 
$$F(\mathbf{x})=\bigoplus_{i\in [m]} h^i \odot \mathbf{x}^{\mathbf{a}^i}=\min\left\{h^i + \langle \mathbf{a}^i, \mathbf{x}\rangle \, |\, i\in [m] \right\},$$
where $\mathbf{x}\in \RR^d$ and $\langle \cdot, \cdot \rangle$ denotes the usual scalar product in $\RR^d$.

The \defn{tropical hypersurface defined by $F$}, or \defn{vanishing locus of $F$} is 
$$\cT(F) := \left\{ \mathbf{x}\in \RR^d\, |\, \text{the minimum of $F(\mathbf{x})$ is attained at least twice}\right\}.$$
It is the image codimension-2-skeleton of the \defn{dome} 
$$\cD(F)=\left\{(\mathbf{x},y)\in \RR^{d+1}\mid \mathbf{x}\in \RR^d, \, y\in \RR, \,  y\leq F(\mathbf{x}) \right\}$$ under the orthogonal projection that omits the last coordinate \cite[Corollary 1.6]{Joswig21}. 
The \defn{cells} of $\cT(F)$ are the projections of the faces of $\cD(F)$ (here we include the regions of $\RR^d$ delimited by $\cT(F)$ as its $d$-dimensional cells ; in fact we are considering the normal complex $NC(F)$ defined in~\cite[after Example 1.7]{Joswig21}).

We say that $\cT(F)$ is the \defn{tropical dual} of the subdivision $\cS$ with admissible function $\height$, since we have the following theorem:

\begin{theorem}[{\cite[Theorem 1.13]{Joswig21}}]\label{thm:tropical_dual}
There is a bijection between the $k$-dimensional cells of $\cS$ and the $(d-k)$-dimensional cells of $\cT(F)$, that reverses the inclusion order. 
\end{theorem}

This bijection sends a vertex $\mathbf{a^j}$ to the region 
$$\left\{\mathbf{x}\in \RR^d \, \middle|\,  \height^j + \langle \mathbf{a}^j, \mathbf{x} \rangle = \min_{i \in [m]} \{ \height^i + \langle \mathbf{a}^i, \mathbf{x} \rangle\} \right\},
$$
and a cell of $\cS$ to the intersection of the regions corresponding to its vertices.

\begin{lemma}\label{lem:tropical_dual_interior}
    The bijection of Theorem~\ref{thm:tropical_dual} restricts to a bijection between the interior cells of $\cS$ and the bounded cells of $\cT(F)$.
\end{lemma}

\begin{proof}
    It is sufficient to show that the bijection restricts to a bijection between the interior facets ($(d-1)$-dimensional cells) of $\cS$ and the bounded edges of $\cT(F)$. 
    Indeed, suppose that it is the case. Any cell of $\cS$ is either maximal and associated to a vertex of $\cT(F)$, or it is an intersection of facets of $\cS$. A non-maximal cell of $\cS$ is interior if and only if it is included only in interior facets of $\cS$. Thus it is sent via the bijection to a cell of $\cT(F)$ that only contains bounded edges. Reciprocally, a non-bounded cell of $\cT(F)$ contains a non-bounded edge, so it is sent to a boundary cell of $\cS$. 

    Let us show the statement about the interior facets of $\cS$ in a fashion similar to the proof of~\cite[Theorem 1.13]{Joswig21}. Let $\widetilde{\cN}(F) := \conv\{(\mathbf{a}^i, r) \, |\, i\in [m], r\geq h^i \}\subseteq \RR^{d+1}$ be the extended Newton polyhedron of $F$, whose lower faces project bijectively onto the cells of $\cS$. 
    Let $\mathbf{e}$ be an edge of $\cT(F)$ and $H$ its corresponding facet in $\cS$ via the bijection. 
    Suppose that $\mathbf{e}$ is unbounded, of the form $\mathbf{e}=\mathbf{w}+\RR_{+}\mathbf{v}$ for some $\mathbf{v},\mathbf{w}\in \RR^d$. Then, for any $\lambda\in \RR_{+}$ the vector $-(\mathbf{w}+\lambda \mathbf{v}, 1)$ is in the normal cone of the lift of $H$ in $\widetilde{\cN}(F)$. Taking the limit of $\lambda\to 0$ of $-\left(\frac{1}{\lambda}\mathbf{w}+\mathbf{v}, \frac{1}{\lambda}\right)$, we obtain that $-(\mathbf{v}, 0)$ is in the normal cone of the lift of $H$, hence $H$ is in the boundary of $\cS$. 
    
    Reciprocally, if $H$ is a boundary facet of $\cS$, it means that the normal cone of the lift of $H$ in $\widetilde{\cN}(F)$ is a two-dimensional cone whose extremal rays can be written $-\RR_{+}(\mathbf{v},0)$ and $-\RR_{+}(\mathbf{w}, 1)$, for some $\mathbf{v},\mathbf{w}\in \RR^d$. For any $\lambda\in \RR_{+}$, the vector $-(\mathbf{w}+\lambda \mathbf{v}, 1)=-\lambda(\frac{1}{\lambda}\mathbf{w}+\mathbf{v}, \frac{1}{\lambda})$ is in this cone, so the point $\mathbf{w}+\lambda \mathbf{v}$ belongs to the edge $\mathbf{e}$ in $\cT(F)$. Hence, this edge is unbounded.
\end{proof}

In the case where $\cA$ is a Cayley embedding, Joswig explains in~\cite[Corollary~4.9]{Joswig21} how the Cayley trick allows us to describe the tropical dual of a regular mixed subdivision with an arrangement of tropical hypersurfaces. 
This extends what was known for triangulations of a product of simplices $\Delta_{m-1}\times\Delta_{d-1}$, which is the Cayley embedding of $m$ copies of the simplex $\Delta_{d-1}$ (the canonical simplex in $\RR^d$) and gives arrangements of tropical hyperplanes, see~\cite[Section~4]{DS04}, \cite{FR15}.

We consider $\cA$ given by the vertices of the Cayley embedding $\cC(P_1, \ldots, P_k)$, with $P_j=\conv(\mathbf{a}^{j, 1}, \ldots, \mathbf{a}^{j, m_j})$ a polytope in $\RR^d$ with integer coordinate vertices, and consider a regular subdivision given by the height $\height=(h^{1,1}, \ldots, h^{1, m_1}, \ldots, h^{k, m_k})\in \RR^{[m_1]\times \ldots \times [m_k]}$.

After the Cayley trick we obtain the subdivision $\widetilde{\cS}$ of the point configuration $\widetilde{\cA}$ given by the points of the form 
$\sum_{j=1}^k \mathbf{a}^{j, i_j}$
 for $(i_1, \ldots, i_k)\in [m_1]\times \ldots \times [m_k]$ with height $h^{(i_1, \ldots, i_k)}=\sum_{j=1}^k h^{j, i_j}$. 

The corresponding tropical polynomial is
\begin{eqnarray*}
\widetilde{F}(\mathbf{x}) &=& \bigoplus_{(i_1, \ldots, i_k)\in [m_1]\times \ldots \times [m_k]} h^{(i_1, \ldots, i_k)}\odot \mathbf{x}^{\sum_{j=1}^k \mathbf{a}^{j, i_j}}\\
&=& \bigoplus_{(i_1, \ldots, i_k)\in [m_1]\times \ldots \times [m_k]} \bigodot_{j=1}^k h^{j, i_j}\odot \mathbf{x}^{\mathbf{a}^{j, i_j}}\\
&=& \bigodot_{j=1}^k \bigoplus_{i_j\in [m_j]} h^{j, i_j}\odot \mathbf{x}^{\mathbf{a}^{j, i_j}}\\
&=& \bigodot_{j=1}^k F_j(\mathbf{x}),\\
\end{eqnarray*}
where $F_j$ is the tropical polynomial $F_j(\mathbf{x}) =  \bigoplus_{i_j\in [m_j]} h^{j, i_j}\odot \mathbf{x}^{\mathbf{a}^{j, i_j}}$.

Then, the vanishing locus $\cT(\widetilde{F})$ is obtained by taking the union of the vanishing loci $\cT(F_j)$ for $j\in [k]$ and the cells of $\cT(\widetilde{F})$ are the intersections of the cells of all $\cT(F_j)$, $ j\in [m]$. We say that these cells are \defn{induced} by the arrangement of tropical hypersurfaces $\left\{\cT(F_j) \, |\, j\in [m]\right\}$.
We have the following theorem as a consequence of Theorem~\ref{thm:tropical_dual}.

\begin{theorem}\label{thm:arrangement_tropical_hypersurfaces}
The tropical dual of the mixed subdivision $\widetilde{\cS}$ is the polyhedral complex of cells induced by the arrangement of tropical hypersurfaces $\left\{\cT(F_j) \, |\, j\in [m]\right\}$.
\end{theorem}

\subsection{The tropical realization}\label{subsec:realization_tropicalhypersurfaces}
Before applying this theorem to our mixed subdivision $\subdivCay$, we explain how to obtain admissible height functions.

\subsubsection{DKK admissible height functions}\label{subsec:height_functions}

Danilov et al.\ provided explicit constructions of admissible height functions for the DKK triangulation of a flow polytope (\cite[Lemma 2 \& 3]{DKK12}) that we can adapt to our particular graph $\Gs$. Note that since their definition of regular subdivisions is in terms of upper faces (linearity areas of a concave function) we change the sign from their $w$ to our $\height$. 
We slightly refine their results.

Let $(G, \preceq)$ be a framed graph. 
Let $P$ and $Q$ be a pair of non-coherent routes of $G$ that are in conflict at subroutes $[x_1, y_1]$, $\dots$, $[x_k, y_k]$, where $x_1\leq y_1 < x_2\leq y_2 \ldots <x_k\leq y_k$ and the subroutes $[x_i, y_i]$ are as long as possible. We define the route $P'$ as the concatenation of subroutes $Px_1$, $x_1Qx_2$, $x_2Px_3$, $\dots$, that we denote $Px_1Qx_2Px_3\dots$ and  $Q'$ the concatenation $Qx_1Px_2Qx_3\dots$. It is clear that $P+Q=P'+Q'$ (where $P+Q$ denotes the union of edges in $P$ and edges in $Q$) and $P'$ and $Q'$ are coherent. We call $P'$ and $Q'$ the \defn{resolvents} of $P$ and $Q$.

We say that there is a \defn{minimal conflict} between routes $P$ and $Q$ if they are in conflict at exactly one subroute $[v_i,v_j]$
and the edges of $P$ and $Q$ that end at $v_i$ are adjacent for the total order $\preceq_{\cI_i}$, (resp. the edges of $P$ and $Q$ that start at $v_j$ are adjacent for the total order $\preceq_{\cO_j}$). 

\begin{figure}[!ht]
    \centering
    \includegraphics[scale=1]{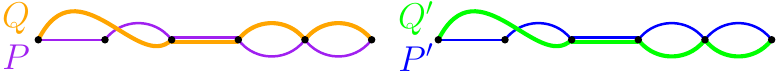}
    \caption{Two routes in conflict (left) and their resolvents (right). 
    }
    \label{fig:crossing_decrossing_routes}
\end{figure}

\begin{lemma}[adaptation of~{\cite[Lemma 2]{DKK12}}]\label{lem:DKKheight1}
Let $(G, \preceq)$ be a framed graph. 
A function $\height$ from the routes of $G$ to $\RR$ is an admissible height function of $\triangDKK$ if and only if: 

For any two non-coherent routes $P$ and $Q$ with resolvents $P'$ and $Q'$ we have:
\begin{equation}\label{eq:admissible_height}
\height(P)+\height(Q)>\height(P')+\height(Q').
\end{equation}
\end{lemma}

\begin{proof}
The original statement of~\cite[Lemma 2]{DKK12} is that a weaker version of this condition, where $P'$ and $Q'$ are not necessarily the resolvents but can be any two routes that satisfy $P+Q=P'+Q'$, is sufficient. Let us show (with the same ideas as their proof) that it is also necessary and that we can even choose $P',Q'$ to be the resolvents of $P$ and $Q$. 
Suppose that $\height$ is an admissible height function of $\triangDKK$ and let $P, Q$ be two non-coherent routes of $(G, \preceq)$ with resolvents $P', Q'$.
Since they form a clique, $P'$ and $Q'$ are the vertices of an edge of the DKK triangulation of $\fpol$. The point $F=\frac{1}{2}(P+Q)=\frac{1}{2}(P'+Q')$ belongs to this edge. Since this edge as to be lifted to a lower face of the lift of the flow polytope $\fpol$ given by the height function $\height$, which is admissible for $\triangDKK$, we necessarily have $\height(P')+\height(Q')<\height(P)+\height(Q)$.
\end{proof}

This statement can be made slightly stronger by restricting condition~\ref{eq:admissible_height} to minimal conflicts.

\begin{lemma}\label{lem:DKKheight}
Let $(G, \preceq)$ be a framed graph. 
A function $\height$ from the routes of $G$ to $\RR$ is an admissible height function of $\triangDKK$ if and only if: 

For any minimal conflict between two routes $P$ and $Q$ with resolvents $P'$ and $Q'$, we have
\begin{equation}
\height(P)+\height(Q)>\height(P')+\height(Q').
\end{equation}
\end{lemma}

\begin{proof}
Let $\height$ be a function from the routes of $G$ to $\RR$ such that for any minimal conflict between two routes $P$ and $Q$ with resolvents $P'$ and $Q'$, we have
$\height(P)+\height(Q)>\height(P')+\height(Q')$. 
It follows from Lemma~\ref{lem:DKKheight1} that we only need to show that for any two non-coherent routes $P$ and $Q$, there exist routes $P'$ and $Q'$ such that $P+Q=P'+Q'$ and $\height(P)+\height(Q)>\height(P')+\height(Q')$. 

First, suppose that $P$ and $Q$ are conflicting at exactly one subroute $[v_i,v_j]$. 
We can build partial routes $R_1=Pv_i, R_2, \ldots, R_k=Qv_i$ that end at $v_i$ and  such that $R_1\prec R_2 \prec \ldots \prec R_k$, their ending edges are adjacent in $\preceq_{\cI_i}$ and they are not in conflict. This can be done by building these partial routes from right to left: the ending edge is determined and we can choose the other ones as we want but if we arrive at a vertex common to a previously built partial route we choose the same edges as in this partial route. Similarly we can build partial routes $S_1=v_jQ, S_2, \ldots, S_t=v_jP$ that start at $v_j$ and such that $S_1\prec S_2 \prec \ldots \prec S_t$, their starting edges are adjacent in $\preceq_{\cO_j}$ and they are not in conflict. 
Then for any $x\in[k-1]$, $y\in [t-1]$ the routes $R_xv_iPv_jS_{y+1}$ and $R_{x+1}v_iPv_jS_{y}$ are in minimal conflict, with resolvents $R_xv_iPv_jS_{y}$ and $R_{x+1}v_iPv_jS_{y+1}$. Hence the condition on $\height$ implies the following inequality:
\begin{equation}\tag{$W_{x,y}$}
\height(R_xv_iPv_jS_{y+1})+\height(R_{x+1}v_iPv_jS_{y}) > \height(R_xv_iPv_jS_{y}) +\height(R_{x+1}v_iPv_jS_{y+1}).
\end{equation}
    When we sum all these inequalities for all $x\in [k-1]$, $y\in[t-1]$ we see that all terms of the form $\height(R_{x}v_iPv_jS_{y})$ are cancelled out by pairs, except for $(x,y)\in \{(1,1), (k,t), (1,t), (k,1)\}$. We end up with:
\begin{equation*}
\height(R_1v_iPv_jS_t)+\height(R_kv_iPv_jS_{1}) > \height(R_1v_iPv_jS_{1}) +\height(R_{k}v_iPv_jS_{t}),
\end{equation*} 
which is exactly $\height(P)+\height(Q)>\height(P')+\height(Q')$, where $P'$ and $Q'$ are the resolvents of $P, Q$.

Now, we can finish the proof by induction on the number of conflicts. Suppose that $\height$ satisfies that for any pair of non-coherent routes $P$ and $Q$ with at most $n$ conflicts their resolvents $P', Q'$ satisfy $\height(P)+\height(Q)>\height(P')+\height(Q')$.
Let $P$ and $Q$ be non-coherent routes with $n+1$ conflicts at subroutes $[x_1, y_1], \ldots, [x_{n+1}, y_{n+1}]$. 
Since the routes $P$ and $Px_1Q$ have $n$ conflicts and their resolvents are $Px_1P'=P'$ and $Px_1Q'$, the induction hypothesis gives us:
\begin{equation*}
    \height(P)+\height(Px_1Q) > \height(P')+\height(Px_1Q').
\end{equation*}
Similarly we have:
\begin{equation*}
    \height(Q)+\height(Qx_1P) > \height(Qx_1P')+\height(Q').
\end{equation*}
Moreover, the routes $P$ and $Qx_1P'$ only have one conflict and their resolvents are $P'$ and $Qx_1P$, so we have
\begin{equation*}
    \height(P)+\height(Qx_1P') > \height(P')+\height(Qx_1P),
\end{equation*}
and similarly:
\begin{equation*}
    \height(Q)+\height(Px_1Q') > \height(Px_1Q)+\height(Q').
\end{equation*}
When we sum up these four inequalities, some terms cancel out and we recover:
\begin{equation*}
    \height(P)+\height(Q) > \height(P')+\height(Q'). \qedhere
\end{equation*}
\end{proof}

Recall that the routes of $\Gs$ are denoted $\route(k, t, \delta)$ as in the discussion after Definition~\ref{def:Gs}.
Adapting {\cite[Lemma 3]{DKK12}} to our context gives us the following lemma.

\begin{lemma}\label{lem:epsilonheight}
Let $\s$ be a composition and $\varepsilon>0$ a sufficiently small real number. 
Consider $\height_{\varepsilon}$ to be the function that associates to a route $\route=\route(k, t_k, \delta)$ of $\Gs$ the quantity
\begin{equation}\label{eq:epsilonheight}
\height_{\varepsilon}(\route)=-\sum_{k\geq c > a \geq 1} \varepsilon^{c-a} (t_c+\delta_a)^2,
\end{equation}
where $t_c=
\begin{cases}
0 &\text{ if } \delta_c=0,\\
s_c &\text{ if } \delta_c=1,
\end{cases}$ for all $c\in [k-1]$.

Then $\height_{\varepsilon}$ is an admissible height function for $\triangDKK[\Gs]$.

\end{lemma}

\begin{proposition}
    In Lemma~\ref{lem:epsilonheight}, it is enough to take $\varepsilon<\frac{1}{n(1+\sum_{j=2}^n (2s_j+1))}$.
\end{proposition}

\begin{proof}
Let $P=\route(k,t,\delta)$ and $Q=\route(k', t', \delta')$ be two routes of $\Gs$ that are in minimal conflict at a common route $[v_{n+1-y}, v_{n-x}]$. We can suppose that $Pv_{n+1-y}\prec Qv_{n+1-y}$. Note that this implies that $\delta_x=1$ and $\delta'_x=0$. We deal separately with the three following cases (which are the only possible ones for a minimal conflict) and compute the quantity $H:=\height_{\varepsilon}(P)+\height_{\varepsilon}(Q)-\height_{\varepsilon}(P')-\height_{\varepsilon}(Q')$. 

\textbf{Case 1: $k=k'=y$, $t\in [s_y-2]$, $t'=t+1$.} 

In the computation of $\height_{\varepsilon}(P)+\height_{\varepsilon}(Q)-\height_{\varepsilon}(P')-\height_{\varepsilon}(Q')$, we see that all pairs $(a,c)$ in formula~\ref{eq:epsilonheight} cancel out either with $\height_{\varepsilon}(P)-\height_{\varepsilon}(Q')$ or $\height_{\varepsilon}(Q)-\height_{\varepsilon}(P')$, except for $(a,c)=(x,y)$.  Thus we have:
\begin{align*}
    H &= \height_{\varepsilon}(P)+\height_{\varepsilon}(Q)-\height_{\varepsilon}(P')-\height_{\varepsilon}(Q')\\
    & = -\varepsilon^{y-x}\Big((t+1)^2+((t+1)+0)^2 - (t+0)^2 - ((t+1)+1)^2 \Big)\\
    &= 2\ \varepsilon^{y-x} >0.
\end{align*}

\textbf{Case 2: $k>k'=y$, $\delta_y=0$, $t'=1$.}

Here the pairs that do not cancel out are all pairs $(a,c)$ for $k\geq c\geq y$ and $x\geq a$.
Then we have:
\begin{align*}
    H &= 
     - \sum_{x\geq a} \varepsilon^{y-a} \Big( \delta_a^2 +(1+\delta'_a)^2 - {\delta'_a}^{2} - (1+\delta_a)^2\Big)
     - \sum_{k\geq c > y, \ x\geq a} \varepsilon^{c-a} \Big( (t_c+\delta_a)^2-(t_c+\delta'_a)^2 \Big)\\
     &= 2\ \varepsilon^{y-x}
     - 2\ \sum_{x> a} \varepsilon^{y-a} ( {\delta'_a} - \delta_a)
     - \sum_{k\geq c > y, \ x\geq a} \varepsilon^{c-a} \Big( 2\ t_c(\delta_a-\delta'_a)+\delta_a^2-{\delta'_a}^{2} \Big)\\
     &\geq 2\ \varepsilon^{y-x}
     - 2\ \sum_{x> a} \varepsilon^{y-a}
     - \sum_{k\geq c > y, \ x\geq a} \varepsilon^{c-a} (2 s_c+1)\\
     &\geq 2\ \varepsilon^{y-x}
     - 2\ \varepsilon^{y-x+1}\Big(x-1 + x \sum_{k\geq c > y} (2 s_c+1) \Big)\\
     &\geq 2\ \varepsilon^{y-x}\Big(1 - \varepsilon\Big(y-2+ (y-1) \sum_{k\geq c > y} (2 s_c+1)\Big)\Big)
\end{align*}

Then, we see that if $\varepsilon<\frac{1}{n(1+\sum_{j=2}^n (2s_j+1))}$, then for any $y\in [2,n]$ we have $$1 - \varepsilon\Big(y-2+ (y-1) \sum_{k\geq c > y} (2 s_c+1)\Big)>0,$$
thus $H>0$.

\textbf{Case 3: $k'>k=y$, $t=s_y-1, \delta'_y=1$.}
    Here again, the pairs that do not cancel out are all pairs $(a,c)$ for $k\geq c\geq y$ and $x\geq a$ and we have:
    \begin{align*}
    H &= 
     - \sum_{x\geq a} \varepsilon^{y-a} \Big( (s_y-1+\delta_a)^2 +(s_y+\delta'_a)^2 - (s_y-1+\delta'_a)^2 - (s_y+\delta_a)^2\Big)\\
     &\phantom{=}
     - \sum_{k\geq c > y, \ x\geq a} \varepsilon^{c-a} \Big( (t'_c+\delta'_a)^2-(t'_c+\delta_a)^2 \Big)\\
     &= 2\ \varepsilon^{y-x} +2\sum_{x>a} \varepsilon^{y-a}(\delta'_a-\delta_a)
     - \sum_{k\geq c > y, \ x\geq a} \varepsilon^{c-a} \Big(2\ t'_c(\delta'_a - \delta_a)+{\delta'_a}^2-\delta_a^2\Big),
     \end{align*}
and the rest of the computations are very similar to the Case 2.
\end{proof}

\subsubsection{Coordinates for the $\s$-permutahedron}

For the remainder of this section $\s$ is assumed to be a composition and $\height$ an admissible height function for $\triangDKK[\Gs]$. 

Since we defined in Section~\ref{subsec:realization_mixedsubdiv} the mixed subdivision $\subdivCay$ from the regular triangulation $\triangDKK[\Gs]$ via the Cayley trick, the following theorem directly follows from Theorem~\ref{thm:arrangement_tropical_hypersurfaces}.

\begin{theorem}\label{thm:arr_trop_hypersurfaces_s-perm}
The tropical dual of the mixed subdivision $\subdivCay$ is the polyhedral complex of cells induced by the arrangement of tropical hypersurfaces 
$$\troparr:=\left\{\cT(F^k_{t}) \, |\, k\in [2, n+1], \, t\in [s_k-1]\right\},$$ where $F^k_{t}(\mathbf{x})=\bigoplus_{} \height(\route(k, t, \delta)) \odot \mathbf{x}^{\delta} = \min \left\{\height(\route(k, t, \delta)) + \sum_{i\in [k-1]} \delta_i x_i \, |\, \delta\in\{0,1\}^{k-1}  \right\}$.
\end{theorem}

\begin{definition}\label{def:s-perm_geom}
We denote by \defn{$\spermgeom$} the polyhedral complex of bounded cells induced by the arrangement $\troparr$.
\end{definition}

\begin{figure}
    \centering
    \includegraphics[scale=0.4]{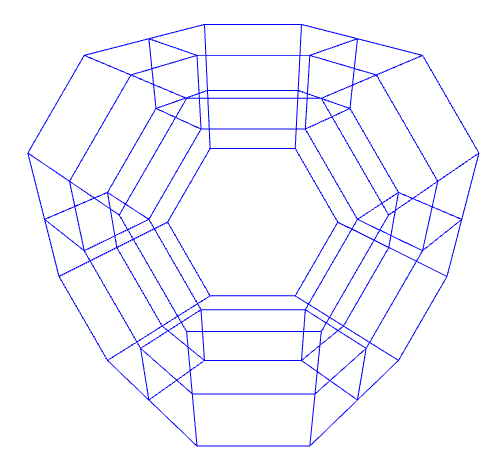}
    \includegraphics[scale=0.4]{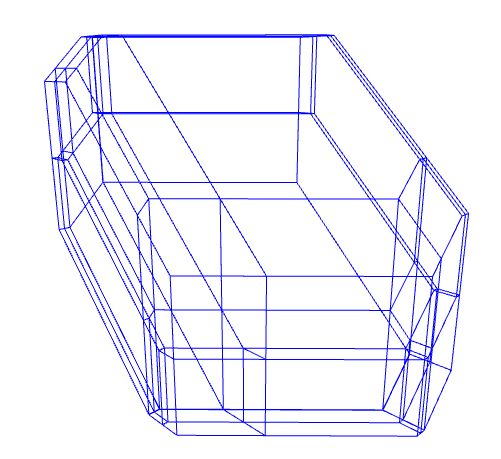}
    \caption{The $(1,1,1,2)$-permutahedron (left) and the $(1,1,1,2)$-permutahedron (right) via their tropical realization.}
    \label{fig:tropical_realizations}
\end{figure}

\begin{theorem}\label{thm:bij_trop_arr}
The face poset of the geometric polyhedral complex $\spermgeom$ is isomorphic to the face poset of the combinatorial $\s$-permutahedron $\spermcombi$.
\end{theorem}

\begin{proof}
 We showed in~Theorem~\ref{thm:bij_mixed_subdiv} that the face poset of $\spermcombi$ is anti-isomorphic to the face poset of interior cells of the mixed subdivision~$\subdivCay$. It then follows from Lemma~\ref{lem:tropical_dual_interior} and Theorem~\ref{thm:arr_trop_hypersurfaces_s-perm} that this poset is isomorphic to the poset of bounded cells of $\troparr$, which is the face poset of $\spermgeom$. 
\end{proof}

Figure~\ref{fig:tropical_realizations} shows some examples of such realizations of the $s$-permutahedron. 

Moreover, we can describe the explicit coordinates of the vertices of $\spermgeom$. 
For a Stirling $\s$-permutation $w$, $a\in [n]$ and $t\in [s_a]$, we denote $i(a^t)$ the length of the prefix of $w$ that precedes the $t$-th occurrence of $a$. 
As explained in the argument leading to Lemma \ref{lem:max_clique},
this prefix is associated to the route $\routep{\prefix{w}{i(a^t)}}$ in the clique
$\Delta_w$. 

\begin{theorem}\label{thm:vertices}
The vertices of $\spermgeom$ are in bijection with Stirling $\s$-permutations. 
Moreover, the vertex $\mathbf{v}(w)=(\mathbf{v}(w)_a)_{a\in[n]}$ associated to a Stirling $\s$-permutation $w$ has coordinates
\begin{equation}\label{eq:coordinates}
    \mathbf{v}(w)_a = \sum_{t=1}^{s_a} \left(\height(\routep{\prefix{w}{i(a^t)}})-\height(\routep{\prefix{w}{i(a^t)+1}})\right).
\end{equation}
\end{theorem}

\begin{proof}
    The bijection between vertices of $\spermgeom$ and Stirling $\s$-permutations is a direct consequence of Theorem~\ref{thm:bij_trop_arr}. 
    
Let $w$ be a Stirling $\s$-permutation.
It is associated via Theorem~\ref{thm:arr_trop_hypersurfaces_s-perm} to the intersection of all regions of the form 
\begin{equation}\label{eq:coordinates_regions}
    \Big\{ \mathbf{x}\in \RR^n \, \Big| \, \height(\route(c, t, \delta)) + \sum_{a\in [c-1]} \delta_a x_a = \min_{\theta \in \{0,1\}^{c-1}} \{ \height(\route(c, t, \theta)) + \sum_{a\in [c-1]} \theta_a x_a\}\Big\},
\end{equation} 
where $\route(c,t,\delta)$ is a route in the clique $\Delta_w$. 
    It follows from the previous remark that this intersection is a single point, that we denote $\mathbf{v}$.
    We show that $\mathbf{v}$ necessarily has the coordinates given by the theorem.
    Let $a\in [n]$. Both routes $\routep{\prefix{w}{i(a^1)}}$ and $\routep{\prefix{w}{i(a^{s_a})+1}}$ are of the form $\route(c,t, \delta)$ and $\route(c,t,\delta')$ respectively, where $c$ is the smallest letter such that the $a$-block is contained in the $c$-block in $w$, and $t$ denotes the number of occurrences of $c$ that precedes the $a$-block. If the $a$-block is contained in no other block we set $c=n+1$ and $t=1$. The indicator vectors $\delta$ and $\delta'$ satisfy that $\delta'-\delta$ is the indicator vector of the letters $b\leq a$ such that the $b$-block is contained in the $a$-block in $w$. 
    The fact that both routes belong to $\Delta_w$ implies that 
    $\height(\routep{\prefix{w}{i(a^1)}})+ \sum_{b\in [c-1]} \delta_b v_b = \height(\routep{\prefix{w}{i(a^{s_a})+1}})+ \sum_{b\in [c-1]} \delta'_b v_b$, thus
    \[ \sum_{\substack{b \in [a] \text{ s.t.}\\\text{$b$-block }\subseteq \text{ $a$-block}}} v_b = \height(\routep{\prefix{w}{i(a^1)}}) - \height(\routep{\prefix{w}{i(a^{s_a})+1}}).\]
    Then, we obtain Equation~\ref{eq:coordinates} by induction on $a$. Indeed, if the equation is true for all $b <a$, then all terms in $\sum_{\substack{b \in [a-1] \text{ s.t.}\\\text{$b$-block }\subset \text{ $a$-block}}} v_b$ cancel by pairs except the terms that correspond to a prefix ending at or just before an occurrence of $a$ in $w$, which are of the form $-\height(\routep{\prefix{w}{i(a^r)}})$ for $r\in [2, s_a]$ or $\height(\routep{\prefix{w}{i(a^r)+1}})$ for $r\in [s_a-1]$.
\end{proof}

\begin{corollary}
    The $\s$-permutahedron $\spermgeom$ is contained in the hyperplane 
    \begin{equation}
    \left\{\mathbf{x}\in \RR^n \, \middle|\, \sum_{i=1}^n x_i = \height(\route(n+1, 1, (0)^n)-\height(\route(n+1, 1, (1)^n))\right\}.
    \end{equation}
\end{corollary}

\begin{theorem}\label{thm:edges} 
    Let $1\leq a < c\leq n$.
    Let $w$ and $w'$ be Stirling $\s$-permutations of the form $u_1B_acu_2$ and $u_1cB_au_2$ respectively, where $B_a$ is the $a$-block of $w$ and $w'$. 

    Then the edge of $\spermgeom$ corresponding to the transposition between $w$ and $w'$ is:
    \begin{equation}\label{eq:edges}
        \mathbf{v}(w')-\mathbf{v}(w)= \left(\height(\routep{u_1c})+\height(\routep{u_1B_a}) - \height(\routep{u_1})-\height(\routep{u_1B_ac}) \right)(\mathbf{e}_a-\mathbf{e}_c),
    \end{equation}
    where $(\mathbf{e}_i)_{i\in [n]}$ is the canonical basis of $\RR^n$.
\end{theorem}

\begin{proof}
    We denote $t:=\#_w(c,a)+1$, so that the transposition from $w$ to $w'$ exchanges the $a$-block with the $t$-th occurrence of $c$.
    We use the expression of the explicit coordinates given in Theorem~\ref{thm:vertices} to compute $\mathbf{v}(w')-\mathbf{v}(w)$. The only routes that do not cancel out are the ones corresponding to prefixes $u_1$, $u_1c$, $u_1B_a$ and $u_1B_ac$, which gives the same route as $u_1cB_a$. Indeed, a prefix contained in $u_1$ is common to $w$ and $w'$ ; a prefix that ends inside the $a$-block does not give information on $c$, so the corresponding route will be common to $w$ and $w'$, and a prefix that ends inside $u_2$ does not give information on the relative order of the $a$-block and letter $c$, so the corresponding route will also be common to $w$ and $w'$. 
    Hence:
    \begin{align*}
        \mathbf{v}(w')-\mathbf{v}(w) &= \left(\mathbf{v}(w')_a-\mathbf{v}(w)_a \right)\mathbf{e}_a + \left(\mathbf{v}(w')_c-\mathbf{v}(w)_c \right)\mathbf{e}_c \\
        &=  \left(\height(\routep{\prefix{w'}{i(a^1)}})-\height(\routep{\prefix{w'}{i(a^{s_a})+1}}) - \height(\routep{\prefix{w}{i(a^1)}}) + \height(\routep{\prefix{w}{i(a^{s_a})+1}}) \right)\mathbf{e}_a \\
        & \phantom{=} + \left(\height(\routep{\prefix{w'}{i(c^t)}})-\height(\routep{\prefix{w'}{i(c^t)+1}}) - \height(\routep{\prefix{w}{i(c^t)}}) + \height(\routep{\prefix{w}{i(c^t)+1}}) \right)\mathbf{e}_c \\
        &= \left(\height(\routep{u_1c})-\height(\routep{u_1cB_a}) - \height(\routep{u_1}) + \height(\routep{u_1B_a}) \right)\mathbf{e}_a \\
        & \phantom{=} + \left(\height(\routep{u_1})-\height(\routep{u_1c}) - \height(\routep{u_1B_a}) + \height(\routep{u_1B_ac}) \right)\mathbf{e}_c\\
        &= \left(\height(\routep{u_1c})+\height(\routep{u_1B_a}) - \height(\routep{u_1})-\height(\routep{u_1B_ac}) \right)(\mathbf{e}_a-\mathbf{e}_c).
    \end{align*}

It follows from Lemma~\ref{lem:DKKheight} that we have 
\[
\height(\routep{u_1c})+\height(\routep{u_1B_a}) - \height(\routep{u_1})-\height(\routep{u_1B_ac}) >0,
\] 
since $P:=\routep{u_1B_a}$ and $Q:=\routep{u_1c}$ are in minimal conflict at $[v_{n+1-c}, v_{n+1-a}]$ and $P':=\routep{u_1}$ and $Q':=\routep{u_1B_ac}$ are their resolvents.
\end{proof}

\begin{lemma}
    For any strictly decreasing sequence of real numbers $\kappa_1>\ldots > \kappa_n$, the direction $\sum_{i=1}^n \kappa_i \mathbf{e}_i$ orients the edges of $\spermgeom$ according to the $\s$-weak order covering relations.
\end{lemma}

\begin{proof}
    This is a direct consequence of Theorem~\ref{thm:edges} and the remark at the end of its proof.
\end{proof}

\begin{lemma}\label{lem:support}
    The support $\supp(\spermgeom)$, i.e.\ the union of faces of $\spermgeom$, is a polytope combinatorially isomorphic to the $(n-1)$-dimensional permutahedron. More precisely it has:
    \begin{enumerate}
        \item vertices $\mathbf{v}(w^ {\sigma})$ for all permutation $\sigma$ of $[n]$ where $w^{\sigma}$ is the Stirling $\s$-permutation 
        \begin{equation*}
        w^{\sigma} = \underbrace{\sigma(1)\ldots \sigma(1)}_{s_{\sigma(1)}\text{ times}}\ldots \underbrace{\sigma(n)\ldots\sigma(n)}_{s_{\sigma(n)}\text{ times}},
        \end{equation*}
        \item facet defining inequalities 
        \begin{equation}\label{eq:halfspace1}
        \langle\delta,\mathbf{x}\rangle \geq \height(\route(n+1, 1, (0)^n))-\height(\route(n+1, 1, \delta)),
        \end{equation}
        \begin{equation}\label{eq:halfspace2}
       \langle\mathbf{1}-\delta,\mathbf{x}\rangle \leq \height(\route(n+1, 1, \delta))-\height(\route(n+1, 1, (1)^n)),
        \end{equation}
        for all $\delta\in \{0,1\}^n$.
    \end{enumerate}
\end{lemma}

\begin{proof}
\begin{enumerate}
    \item Let $\sigma$ be a permutation of $[n]$. we consider the linear functional $f(x)=\sum_{a\in[n]} \sigma(a) x_a$. Among the vertices of $\spermgeom$ $f$ is maximized on $\mathbf{v}(w^{\sigma})$. 
    Indeed, let $w'$ be a Stirling $\s$-permutation. 
    \begin{itemize}
    \item If $w'$ contains an ascent $(a,c)$ such that $\sigma(a)>\sigma(c)$, then $f$ is increasing along the edge of direction $\mathbf{e}_a-\mathbf{e}_c$ corresponding to the transposition of $w'$ along the ascent $(a,c)$.
    \item If $w'$ contains a descent $(c',a')$ such that $\sigma(a')<\sigma(c')$, then $f$ is increasing along the edge of direction $\mathbf{e_{c'}}-\mathbf{e_{a'}}$ corresponding to the transposition of $w'$ along the descent $(c',a')$.
    \item If $w'$ is in neither of the above cases, than necessarily $w'=w^{\sigma}$.
    \end{itemize}
    This shows that the vertices of $\supp(\spermgeom)$ have the same normal cones as the $(n-1)$-permutahedron (embedded in $\RR^n$), hence its normal fan is the braid fan. 
    \item It follows from the fact that all cliques $\Delta_w$ contain the routes $\route(n+1, 1, (0)^n)$ and $\route(n+1, 1, (1)^n)$ 
    (see the remark before Lemma \ref{lem:max_clique})
    that all vertices of $\spermgeom$ are contained in the region 
    \begin{align*}
     \Big\{ \mathbf{x}\in \RR^n \, \Big| \, \height(\route(n+1, 1, (0)^n)) &= \height(\route(n+1, 1, (1)^n)) + \sum_{a\in [n]} x_a \\
     &= \min_{\theta \in \{0,1\}^{n}} \{ \height(\route(n+1, 1, \theta)) + \sum_{a\in [n]} \theta_a x_a\}\Big\}.  
    \end{align*}
    This region is exactly defined by intersecting the half-spaces defined by \ref{eq:halfspace1} and \ref{eq:halfspace2}. 

    Moreover, let $\delta\in \{0,1\}^n$ and $I:=\{i\in [n]\, |\, \delta_i=1\}$. The equality in \ref{eq:halfspace1} and \ref{eq:halfspace2} is achieved exactly on vertices $\mathbf{v}(w^{\sigma})$ where $\{\sigma(1), \ldots, \sigma(|I|)\}=I$. Hence, these inequalities define the facets of $\supp(\spermgeom)$.
\end{enumerate}
    
\end{proof}

\begin{remark}
    With similar arguments we can see that the restriction of the $\s$-weak order to a face of $\supp(\spermgeom)$, associated to an ordered partition, will correspond to a product of $\s'$-weak orders, one for each part of the ordered partition.
\end{remark}

Note that Lemma~\ref{lem:support} finishes to answer Conjecture~\ref{conj:s-permutahedron} in the case where $\s$ is a composition, because then the zonotope $\sum_{1\leq i<j\leq n}s_j[\mathbf{e}_i,\mathbf{e}_j]$ is combinatorially isomorphic to the $(n-1)$-dimensional permutahedron. 

In the case where $\height$ is given by Lemma~\ref{lem:epsilonheight}, we can even go a bit further.

\begin{proposition}
    Let $\varepsilon>0$ be a small enough real number so that $\height_{\varepsilon}$ is an admissible height function for $\triangDKK[\Gs]$. 

    Then the support $\supp(\spermgeom[\height_{\varepsilon}])$ is a translation of the zonotope 
    $2\sum_{1\leq a < c\le n} s_c\  \varepsilon^{c-a}[\mathbf{e}_a, \mathbf{e}_c]. $
\end{proposition}

\begin{proof}
    It follows from Lemma~\ref{lem:support} that the edges of $\supp(\spermgeom[\height_{\varepsilon}])$ are of the form $[\mathbf{v}(w^{\sigma}), \mathbf{v}(w^{\sigma'})]$, where $\sigma$ and $\sigma'$ are permutations of $[n]$ related by a transposition along an ascent $(a,c)$. When we plug the expression~\ref{eq:epsilonheight} of $\height_{\varepsilon}$ into the formula~\ref{eq:edges}, where the letter $c$ is replaced by $s_c$ occurrences of $c$, we see that the only terms that do not cancel out are those involving the pair $(a,c)$:
    \begin{align*}
    \mathbf{v}(w^{\sigma'})- \mathbf{v}(w^{\sigma}) &= -\varepsilon^{c-a}\Big( (0+1)^2 + (s_c+0)^2 - (0+0)^2 - (s_c+1)^2\Big)(\mathbf{e}_a - \mathbf{e}_c)\\
    &= 2\ s_c \ \varepsilon^{c-a}(\mathbf{e}_a - \mathbf{e}_c).
    \end{align*}
    Hence all edges of the same direction have the same length, and since $\supp(\spermgeom[\height_{\varepsilon}])$ is combinatorially equivalent to a permutahedron, it follows that it is a zonotope.
\end{proof}

\section{A Lidskii-type decomposition of the \texorpdfstring{$\s$}{s}-permutahedron}\label{sec:lidskii_main}

In this section we apply the Lidskii formulas for the volume and lattice points of flow polytopes to the polytope $\fpol[\Gs]$ to obtain two identities for the number of $\s$-decreasing trees. We also interpret the identities enumeratively by partitioning the set of trees accordingly. Lastly, we illustrate these identities in the context of the $\s$-permutahedron.

\subsection{The Lidskii formula and enumerations of \texorpdfstring{$\s$}{s}-decreasing trees}\label{subsec:lidskii}

Given two weak compositions ${\bf b}=(b_0,\ldots,b_{n-1})$ and ${\bf c}=(c_0,\ldots,c_{n-1})$ of $N$, we say that ${\bf b}$ \defn{dominates} ${\bf c}$, denoted by ${\bf b} \succeq {\bf c}$, if $\sum_{i=1}^k b_i \geq \sum_{i=1}^k c_i$ for all $k\in\{0,\ldots,n-1\}$. Let $\bbinom{m}{k}=\binom{m+k-1}{k}$ be the number of multisets of $[m]$ of size $k$.

\begin{theorem}[{Lidskii lattice point formulas \cite{BVkpf,MM19}}] \label{thm: lidskii lattice point formulas} 
Given a graph $G$ on~$\{v_0,\ldots,v_n\}$ with $m$ edges and $\bf a$ $=(a_0, \dots, a_n)$ such that $\sum_i a_i=0$ and $a_i \in \mathbb{N}$ for $i=0,\ldots,n-1$, then
\begin{align}
    \#\mathcal{F}_G^{\mathbb{Z}}(\textup{\textbf{a}}) &= \sum_{\bf j} \binom{a_0+t_0}{j_0} \binom{a_1+t_1}{j_1}\cdots \binom{a_{n-1}+t_{n-1}}{j_{n-1}} \cdot \# \mathcal{F}^{\mathbb{Z}}_G({\bf j}-{\bf t}), \label{eq: 1st lidskii lattice points} \\
    &= \sum_{\bf j} \bbinom{a_0-d_0}{j_0} \bbinom{a_1-d_1}{j_1}\cdots \bbinom{a_{n-1}-d_{n-1}}{j_{n-1}} \cdot \# \mathcal{F}^{\mathbb{Z}}_G({\bf j}-{\bf t}) \label{eq: 2nd lidskii lattice points} , 
\end{align}
where all sums are over weak compositions ${\bf j}=(j_0,\ldots,j_{n-1})$ of $m-n$ that are $\succeq {\bf t}:=(t_0,t_1,\ldots,t_{n-1})$ in dominance order, $t_i=\outdeg_G(v_i)-1$, $d_i=\indeg_G(v_i)-1$, and $\bbinom{n}{k}=\binom{n+k-1}{k}$.
\end{theorem}

\begin{corollary} \label{cor:identitise s-trees}
For a (weak) composition $s=(s_1,\ldots,s_n)$, the number of elements of the $\s$-weak order decomposes as
\begin{align}
\prod_{i=1}^{n-1}\Bigl(1+\sum_{r=n-i+1}^n s_r\Bigr)
&= 
\sum_{\bf j}  \binom{s_n+1}{j_1}\binom{s_{n-1}+1}{j_2} \cdots \binom{s_2+1}{j_{n-1}} \cdot \prod_{i=1}^{n-1} (j_1+\cdots + j_i-i+1). 
    \label{eq: first lidskii formula number of trees}\\
    &\,=\, \sum_{\bf j} \bbinom{s_n+1}{j_1}\bbinom{s_{n-1}-1}{j_2} \cdots \bbinom{s_2-1}{j_{n-1}} \cdot \prod_{i=1}^{n-1} (j_1+\cdots + j_i-i+1)  \label{eq: second lidskii formula number of trees},
\end{align}
where all sums are over weak compositions ${\bf j}=(j_1,j_2,\ldots,j_{n-1})$ of $n-1$ that are $\succeq (1,1,\ldots,1)$.
\end{corollary}

We give two proofs of this result, one geometric and the other combinatorial.

\begin{proof}[First proof]
For a weak composition $\s$, the flow polytope $\fpol[\oruga](s_n,s_{n-1},\ldots,s_2,-\sum_i s_i)$ is integrally equivalent to the box polytope  $\prod_{i=1}^{n-1}[0,\sum_{r=n-i+1}^n s_r]$. Moreover,  the number of lattice points of $\fpol[\oruga](s_n,s_{n-1},\ldots,s_2,-\sum_i s_i)$ coincides with the number of $\s$-decreasing trees from Equation~\eqref{eq: formula s decreasing trees}, so
\begin{equation} \label{eq: int flows box s-trees}
\#\fpol[\oruga]^{\mathbb{Z}}(s_n,s_{n-1},\ldots,s_2,-\sum_i s_i) = \prod_{i=1}^{n-1}(1+s_{n-i+1}+s_{n-i+2}+\cdots + s_n) = \# \settrees.
\end{equation}
Next, we use the Lidskii formulas in Theorem~\ref{thm: lidskii lattice point formulas} to calculate the LHS above.  
The graph $\oruga$ has shifted outdegrees $o_i=1$  and shifted indegrees $d_i=1$, for $i=1,\ldots,n$. By Equations~\eqref{eq: 1st lidskii lattice points} and \eqref{eq: 2nd lidskii lattice points} we have that 
\begin{align*}
\#\fpol[\oruga]^{\mathbb{Z}}&\Bigl(s_n,s_{n-1},\ldots,s_2,-\sum_i s_i\Bigr)\\
&=  \sum_{\bf j}  \binom{s_n+1}{j_1}\binom{s_{n-1}+1}{j_2} \cdots \binom{s_2+1}{j_{n-1}} \# \mathcal{F}^{\mathbb{Z}}_{\oruga}({\bf j}-{\bf 1}),\\
&=  \sum_{\bf j} \bbinom{s_n+1}{j_1}\bbinom{s_{n-1}-1}{j_2} \cdots \bbinom{s_2-1}{j_{n-1}} \#\mathcal{F}^{\mathbb{Z}}_{\oruga}({\bf j}-{\bf 1}),
\end{align*}
where the sums are over compositions ${\bf j}=(j_1,\ldots,j_{n-1})$ of $n-1$ that dominate $(1,\ldots,1)$. Lastly, we count the integer flows in  $\mathcal{F}^{\mathbb{Z}}_{\oruga}(j_1-1,j_2-1,\ldots,j_{n-1}-1)$. For such an integer flow,  the incoming flow to vertex $i=1,\ldots,n-1$ is $j_1+\cdots+j_{i-1}-(i-1)$. Since the netflow on vertex $i$ is $j_i-1$ then the outgoing flow is $j_1+\cdots+j_{i}-i$. Moreover, there are $j_1+\cdots+j_i-i+1$ possible outgoing integer flows on the two edges $(i,i+1)$. Since the graph $\oruga$ is a path then each such choice is independent. Thus, in total we have 
\[
\#\mathcal{F}^{\mathbb{Z}}_{\oruga}(j_1-1,j_2-1,\ldots,j_{n-1}-1) = \prod_{i=1}^{n-1} (j_1+\cdots + j_i-i+1),
\]
as desired and the result follows.
\end{proof}

\begin{proof}[Second proof]
These two formulas can also be obtained by purely combinatorial double counting of $\s$-decreasing trees or Stirling $\s$-permutations. 

The first formula \eqref{eq: first lidskii formula number of trees}, which is valid even if some $s_i$ are zero, can be obtained in the following way from building $\s$-decreasing trees. We begin with the node labelled $n$. Then, at step $1$ we choose which of its $s_n+1$ children will be nodes (as opposed to leaves). This gives a coefficient $\binom{s_n+1}{j_1}$, where $j_1\geq 1$. Among these $j_1$ nodes we choose one to carry the label $n-1$. At the beginning of step $i$, we have a partial $\s$-decreasing trees with $i$ nodes labelled $n$, $n-1$, $\ldots$, $n+1-i$ and $j_1 + \ldots + j_{i-1} - (i-1)$ nodes without label, where $j_k$ is the number of non-empty subtrees of the node $n+1-k$, whose positions were chosen at the step $k$. At step $i$ we choose which of the $s_{n+1-i}+1$ subtrees will be nodes and denote $j_i$ their number. Then we have $n_{i}:=j_1 + \ldots + j_{i} - (i-1)$ nodes without labels. We have to ensure that $n_i>0$, which means the $j_k$'s are such  that $\sum_{k=1}^i j_i \geq i$. Then we choose one of these $n_i$ nodes to carry the label $n-i$.
We stop after step $n-1$. 

Similarly, the second formula \eqref{eq: first lidskii formula number of trees} can be obtained by the following way of building a Stirling $\s$-permutation.
Recall that for a multipermutation $w$ and a letter $a$, the $a$-block of $w$ is the shortest substring of $w$ that contains all occurrences of $a$.
A \defn{block} of $w$ is any $a$-block for $a\in \{1, \ldots, n\}$. We say that an $a$-block covers a $b$-block in $w$ if $a$ is the smallest letter such that the $a$-block of $w$ contains the $b$-block of $w$. 
To build a Stirling $\s$-permutation, we can start with the $s_n$ occurrences of $n$ and choose a number $j_1$ of blocks that will be covered by the $n$-block or appear before or after it in the final multipermutation $w$. There are $\bbinom{s_n+1}{j_1}$ ways to arrange these blocks among the occurrences of $n$. Then we choose which of these blocks will be the $(n-1)$-block. There are $j_1$ possibilities. At the beginning of step $i>1$, we have a partial Stirling $\s$-permutation that contains all occurrences of the letters from $n$ to $n+1-i$ and $j_1 + \ldots + j_{i-1} - (i-1)$ unlabelled blocks, where $j_k$ is the number of blocks covered by the $(n+1-k)$-block, whose positions were chosen at the step $k$. At step $i$ we choose the number $j_i$ of blocks that will be covered by the $(n+1-i)$-block and one among the $\bbinom{s_{n+1-i}-1}{j_i}$ ways to arrange them between the first and the last occurrence of $n+1-i$. We choose one of the $n_{i}:=j_1 + \ldots + j_{i} - (i-1)$ unlabeled blocks to be the $(n-i)$-block.
We stop after step $n-1$. 
\end{proof}

\begin{example} \label{ex: case n=3 lidskii}
For $n=3$, Corollary~\ref{cor:identitise s-trees} yields 
\begin{align}
(1+s_3)(1+s_2+s_3) 
&= \binom{s_{3}+1}{1}\binom{s_{2}+1}{1}+ \binom{s_{3}+1}{2}\cdot 2
\label{eq: 1st lidskii n=3}\\
&=\binom{s_{3}+1}{1}\binom{s_2-1}{1}+\binom{s_{3}+2}{2}\cdot 2. 
\label{eq: 2nd lidskii n=3}
\end{align}
\end{example}

\begin{example}\label{ex: case n=4 lidskii}
For $n=4$, Corollary~\ref{cor:identitise s-trees} yields
\begin{align*}
    (1+&s_4)(1+s_3+s_4)(1+s_2+s_3+s_4) \\
    &=
\hbox{$\binom{s_{4}+1}{1}\binom{s_{3}+1}{1}\binom{s_2+1}{1}
 + \binom{s_4+1}{1}\binom{s_{3}+1}{2}\cdot 2 +\binom{s_{4}+1}{2} \binom{s_{2}+1}{1}\cdot 2+\binom{s_{4}+1}{2}\binom{s_3+1}{1}\cdot 4+\binom{s_{4}+1}{3}\cdot 6$}\\
&=
\hbox{$\binom{s_{4}+1}{1}\binom{s_{3}-1}{1}\binom{s_{2}-1}{1}
+\binom{s_{4}+1}{1} \binom{s_{3}}{2}\cdot 2
+ \binom{s_{4}+2}{2} \binom{s_{2}-1}{1}\cdot 2+\binom{s_{4}+2}{2} \binom{s_{3}-1}{1}\cdot 4+\binom{s_{4}+3}{3}\cdot 6$}.\nonumber
\end{align*}
\end{example}

\begin{remark} \label{rem: number of terms}
The weak compositions ${\bf j}=(j_1,\ldots,j_n)$ of $n-1$ that dominate $(1,\ldots,1)$ is one family of objects enumerated by the Catalan numbers~\cite[Problem 80]{CatBook}. There are $C_{n-1}:=\frac{1}{n}\binom{2n-2}{n-1}$ such compositions, and thus that many terms in \eqref{eq: 1st lidskii lattice points} and \eqref{eq: 2nd lidskii lattice points}, respectively. 
\end{remark}

\begin{remark} \label{rem: connection volume and lattice point oruga}
When $\s$ is a composition,  by Corollary~\ref{cor:volume is number of trees}, we have that the RHS of \eqref{eq: first lidskii formula number of trees} gives the volume of $\fpol[\Gs]$. Indeed in this case, 
by Corollary~\ref{cor:volume 10...0-1 case} we have that 
\begin{align}
\vol \fpol[\Gs](1,0,\ldots,0,-1) &= \#\fpol[\Gs]^{\mathbb{Z}}(0,s_n, s_{n-1},\ldots, s_2,-\sum_i s_i) \notag 
\\
&= \#\fpol[\oruga]^{\mathbb{Z}}(s_n,s_{n-1},\ldots,s_2,-\sum_i s_i),\label{eq: vol Gs as int flows in Ps}
\end{align}
where the second equality follows since the netflow on the first vertex of $\Gs$ is zero and so the edges of vertex $v_0$ do not contribute. Thus, we can view \eqref{eq: first lidskii formula number of trees} and \eqref{eq: second lidskii formula number of trees} as decomposition formulas for the volume. This is the approach used by Kapoor--M\'esz\'aros--Setiabrata in \cite{KMS} to prove Equation~\eqref{eq: 2nd lidskii lattice points}.
\end{remark}

\begin{remark}
Note that \eqref{eq: first lidskii formula number of trees} and \eqref{eq: second lidskii formula number of trees} hold when some $s_i$ are zero. Moreover, each term of the RHS of the first formula \eqref{eq: first lidskii formula number of trees} is nonnegative for $s_i\geq 0$, however terms of the RHS of the second formula can be negative. See Examples~\ref{ex: case n=3 lidskii} and~\ref{ex: case n=4 lidskii} for $\s=(1,0,1)$ and $\s=(0,1,0,2)$, respectively.   
\end{remark}

\subsection{Towards a Lidskii-type decomposition of the \texorpdfstring{$\s$}{s}-permutahedron}

In the previous section we used the Lidskii formulas to give an enumerative decomposition for the number of $\s$-decreasing trees, which index the vertices of the $\s$-permutahedron $\spermcombi$. 
We note that the enumerative decomposition~\eqref{eq: second lidskii formula number of trees} can be translated into a partition of the Hasse diagram of the $\s$-weak order into intervals. 
For every $\bf j$ there are $\prod_{i=1}^{n-1} (j_1+\dots+j_i-i+1)$ pieces that are products of the Gale orders on $\bbinom{[s_n+1]}{j_1}$, $\bbinom{[s_{n-1}-1]}{j_2}$, $\dots$, $\bbinom{[s_2-1]}{j_{n-1}}$, where $\bbinom{[m]}{k}$ denotes the collection of multisets of $[m]$ of size $k$. 

The \defn{Gale order} on $\bbinom{[m]}{k}$ is given by $A\leq B$ if  $a_i\leq b_i$ for all $i\in[k]$, where $a_1, \dots, a_k$ (respectively $b_1, \dots, b_k$) are the elements of $A$ (respectively $B$) ordered increasingly (see \cite{Gale}). We denote this by $\Gale$.

See Figure~\ref{fig:subdivision_trees} for the example of $s=(1,3,2)$. 
The central blue piece corresponds to ${\bf j} = (1,1)$, which is the product of $\Gale[1, 3]$ and $\Gale[1, 2]$ where each is a chain of respective sizes $3$ and $2$. 
The two side pieces each correspond to ${\bf j}=(2,0)$, and they are the Hasse diagrams of $\Gale[2,3]$.

\begin{figure}    \includegraphics[scale=0.7]{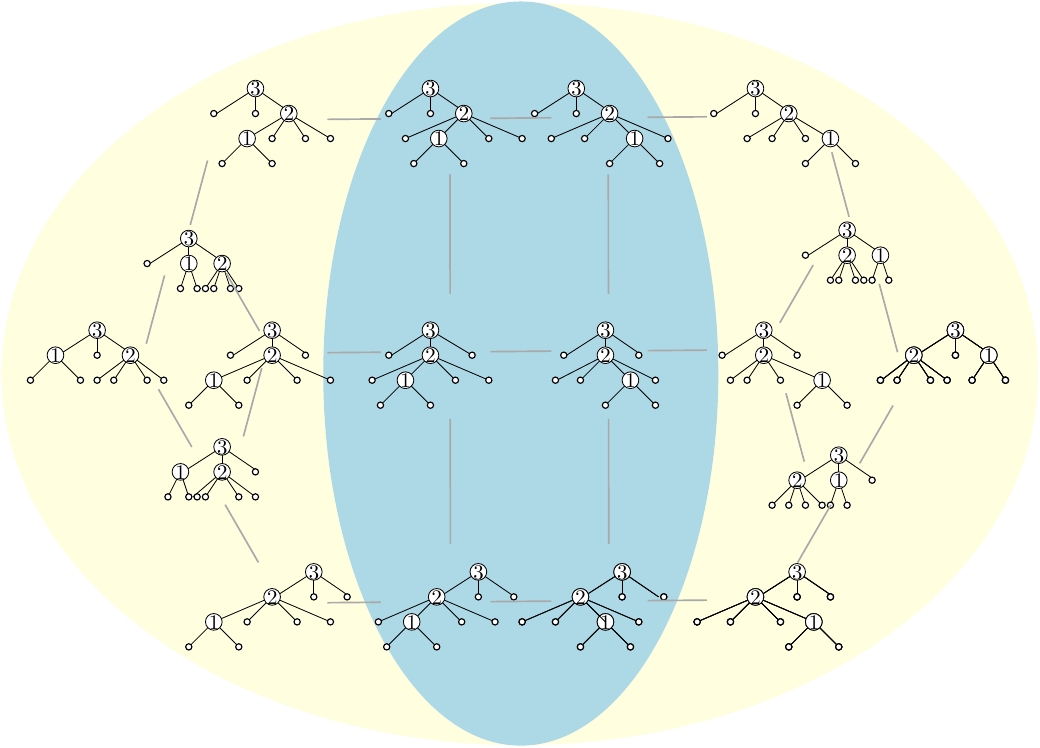}
\caption{Subdivision of the complex of $\s$-trees for $s=(1,3,2)$ that yields Equation~\eqref{eq: 2nd lidskii n=3}.}
    \label{fig:subdivision_trees}
\end{figure}

\section{Further directions}\label{sec:further}

\subsubsection*{Relation with the \texorpdfstring{$\s$}{s}-associahedron}

Ceballos and Pons also conjectured (\cite[Conjecture 2]{CP20}) that there exists a geometric realization of the $\s$-permutahedron (when $\s$ is a strict composition) such that the $\s$-associahedron can be obtained from it by removing certain facets. 
Our realizations seem very promising for providing a geometric relation between $\s$-permutahedra and $\s$-associahedra but this is still work in progress.

\subsubsection*{Other poset structures on \texorpdfstring{$\s$}{s}-decreasing trees}

In this article, we considered a specific triangulation of the flow polytope $\fpol[\Gs]$ induced by a fixed framing of $\Gs$ described in Definition~\ref{def:Gs}.
It would be interesting to study DKK triangulations of $\fpol[\Gs]$ arising from other framings. For example, Bell et al.~\cite{vBGDLMCY21} showed that both the $\s$-Tamari lattice and the principal order ideal $I(\nu)$ in Young's lattice, where $\nu_i=1+s_{n-i+1}+s_{n-i+2}+\cdots+s_n$, can be realized as the graph dual to DKK triangulations of a flow polytope of the $\nu$-caracol graph.
Do other interesting posets on $\s$-decreasing trees arise from other framed triangulations of $\fpol[\Gs]$?

\subsubsection*{The case where \texorpdfstring{$\s$}{s} is a weak composition}\label{subsec:further_weak_composition}

The $\s$-weak order of Ceballos and Pons is defined for weak compositions but our realizations, which rely on a triangulation of the flow polytope $\mathcal{F}_{\Gs}$ (see Sections~\ref{subsec:realization_flowpolytopes}, \ref{subsec:realization_mixedsubdiv}, and \ref{subsec:realization_tropicalhypersurfaces}) hold only for strict compositions, since the graph $\Gs$ is not defined for weak compositions. 
Furthermore, although Corollary~\ref{cor:volume is number of trees} does not hold for weak compositions, Equation~\eqref{eq: int flows box s-trees} does, so it would be interesting to extend our story to the case of weak compositions.

\section{Acknowledgements}

This work began under the collaborative project MATH-AMSUD with code 22-MATH-01 ALGonCOMB and we are very grateful for their support.
We thank Viviane Pons for helpful comments and suggestions and for proposing this problem during the open problem session of the VIII Encuentro Colombiano de Combinatoria ECCO 2022. We extend our gratitude to all of the organizers of ECCO 2022. We also thank Cesar Ceballos, Balthazar Charles, Arnau Padrol,  Vincent Pilaud, Germain Poullot, Francisco Santos, Yannic Vargas, and the combinatorics team of LIGM for helpful comments and proofreading of previous versions of this manuscript. For this work we used the open-source software \texttt{Sage}~\cite{Sage}.
Rafael S. Gonz\'alez D'Le\'on is very grateful for the summer research stipend program of Loyola University Chicago since part of this work happened under their support.

\printbibliography

\end{document}